\documentclass[draft]{amsart}
\usepackage{amssymb}
\usepackage[abbrev]{amsrefs}
\numberwithin{equation}{section}
\def\N{\mathbb N}
\def\R{\mathbb R}

\providecommand{\norm}[1]{\left\lVert#1\right\rVert}

\providecommand{\abs}[1]{\left\lvert#1\right\rvert}

\DeclareMathOperator*{\Argmin}{argmin}
\DeclareMathOperator*{\Argmax}{argmax}
\DeclareMathOperator{\Fix}{\mathcal{F}}
\DeclareMathOperator{\AC}{\mathcal{A}}

\DeclareMathOperator{\Saddle}{\mathcal{SP}}
\DeclareMathOperator{\Co}{\mathrm{co}}
\newcommand{\CAT}{\textup{CAT}}

\theoremstyle{plain}
\newtheorem{theorem}{Theorem}[section]
\newtheorem{lemma}[theorem]{Lemma}

\newtheorem{corollary}[theorem]{Corollary}

\theoremstyle{definition}
\newtheorem{definition}[theorem]{Definition}

\theoremstyle{remark}
\newtheorem{remark}[theorem]{Remark}
\allowdisplaybreaks
\title[] 
{Sion's minimax theorem and the proximal point algorithm in Hadamard spaces}
\author[F.~Kohsaka]{Fumiaki~Kohsaka}
\address[F.~Kohsaka]
{Department of Mathematical Sciences, Tokai University, 
Kitakaname, Hiratsuka, Kanagawa 259-1292, Japan}
\email{f-kohsaka@tokai.ac.jp}
\subjclass[2010]{}
\date{\today; File: \jobname}
\begin{document}
\dedicatory{Dedicated to the memory of my parents,\\
Yuji Kohsaka (1940--2024) and Yukiko Kohsaka (1947--2023)}

\begin{abstract}
 We obtain Sion's minimax theorem in Hadamard spaces and 
 discuss its applications. Among other things, we study 
 several fundamental properties of resolvents of 
 saddle functions in Hadamard spaces. An application to 
 the proximal point algorithm for minimax problems 
 in Hadamard spaces are also included. 
\end{abstract}
\maketitle

\section{Introduction}

The aim of this paper is to study minimax problems in Hadamard spaces. 
We first obtain Sion's minimax theorem in Hadamard spaces. 
We next study several fundamental properties of resolvents of 
saddle functions in Hadamard spaces. We finally obtain 
existence and convergence theorems on the proximal point algorithm 
for minimax problems in Hadamard spaces.  

Minimax problems are one of fundamental problems in 
nonlinear analysis, nonlinear optimization, game theory, and so on. 
This is a problem of finding a point 
$(x_0,y_0)$ in a product set $X\times Y$ such that 
\begin{align}
 f(x,y_0)\leq f(x_0,y_0)\leq f(x_0,y)
\end{align}
for all $(x,y)\in X\times Y$, where 
$f\colon X\times Y\to \R$ is a given function. 
The set of all saddle points of $f$ is denoted by 
$\Saddle (f)$. 
Defining a real function $F$ by 
\begin{align}
F\bigl((x,y), (x',y')\bigr) = f(x,y') -f(x',y) 
\end{align}
for all $(x,y)$ and $(x',y')$ in $X\times Y$, 
we know that $(x_0,y_0)$ is an element of $\Saddle(f)$ 
if and only if it is an equilibrium point of $F$ in the 
sense of Blum and Oettli~\cite{MR1292380}, i.e., 
\begin{align}
 F\bigl((x_0,y_0), (x,y)\bigr) \geq 0
\end{align}
for all $(x,y)\in X\times Y$. 

In 1958, Sion~\cite{MR0097026} proved the following minimax theorem 
in Hausdorff topological vector spaces; 
see also~\cite[Theorem~16]{MR0399979} and~\cite[Theorem~6.3.3]{MR1864294}: 

 \begin{theorem}[Sion~{\cite[Theorem~3.4]{MR0097026}}]\label{thm:Sion-minimax}
 Let $E$ and $F$ be real Hausdorff topological vector spaces, 
 $X$ a nonempty convex subset of $E$, 
 $Y$ a nonempty compact convex subset of $F$, 
 and $f\colon X\times Y\to \R$ a function satisfying 
 \begin{enumerate}
  \item[(i)] $f(\,\cdot\,,y)$ is upper semicontinuous 
 and quasi-concave for each $y\in Y$; 
  \item[(ii)] $f(x,\cdot \,)$ is lower semicontinuous 
 and quasi-convex for each $x\in X$. 
 \end{enumerate}
 Then the equality 
 \begin{align}
  \min_{y\in Y} \sup_{x\in X} f(x,y) = \sup_{x\in X}\min_{y\in Y} f(x,y)
 \end{align}
 holds. 
\end{theorem}

If the set $X$ is also compact in Theorem~\ref{thm:Sion-minimax}, 
the following minimax equality 
 \begin{align}
  \min_{y\in Y} \max_{x\in X} f(x,y) = \max_{x\in X}\min_{y\in Y} f(x,y)
 \end{align}
holds. This implies that $\Saddle (f)$ is nonempty. 

In 1988, Komiya~\cite{MR0930413} found an elementary proof of Sion's minimax theorem 
in topological vector spaces without using either Brouwer's fixed point theorem or 
KKM theorem~\cites{KKM-FundMath29, MR2184742}. 

The proximal point algorithm first introduced by 
Martinet~\cite{MR0298899} is a well known method for 
approximating a solution to the equation 
\begin{align}\label{eq:zero}
 0\in Au 
\end{align} 
for a maximal monotone operator $A\colon H\to 2^H$ in a real 
Hilbert space $H$. The set of all points $u\in H$ 
satisfying~\eqref{eq:zero} is denoted by $A^{-1}(0)$. 
This algorithm generates a sequence $\{x_n\}$ 
by $x_1\in H$ and 
\begin{align}
 0 \in Ax_{n+1} + \frac{1}{\lambda_n} (x_{n+1}-x_n)
\end{align}
for all $n\in \N$, where $\{\lambda_n\}$ is a sequence of positive real numbers. 
In 1976, Rockafellar~\cite{MR0410483} proved 
under the condition $\inf_n\lambda_n >0$ that 
$\{x_n\}$ is bounded if and only if $A^{-1}(0)$ is nonempty; 
if $A^{-1}(0)$ is nonempty, then $\{x_n\}$ is 
weakly convergent to an element of $A^{-1}(0)$. 

In 1968, using the concept of subdifferentials of 
convex functions, Rockafellar~\cite{MR0285942} constructed 
maximal monotone operators associated with 
saddle functions in Banach spaces. 
Further, Rockafellar~\cite{MR0410483} applied the proximal point algorithm to 
the problem of finding saddle points in Hilbert spaces 
as follows; see also~\cite{MR2548424}: 
\begin{theorem}[Rockafellar~\cite{MR0410483}]
\label{thm:Rockafellar-ppa-saddle}
 Let $X$ and $Y$ be nonempty closed convex subsets of 
 real Hilbert spaces $H_1$ and $H_2$, respectively, 
 $f\colon X\times Y\to \R$ a function satisfying 
 \begin{enumerate}
  \item[(i)] $f(\,\cdot\,,y)$ is upper semicontinuous 
 and concave for each $y\in Y$; 
  \item[(ii)] $f(x,\cdot \,)$ is lower semicontinuous 
 and convex for each $x\in X$, 
 \end{enumerate} 
 and $\{(x_n,y_n)\}$ a sequence in $X\times Y$ defined by 
 $(x_1,y_1)\in X\times Y$ and $(x_{n+1},y_{n+1})$ 
 is the unique saddle point of the function 
 \begin{align}
  (z,w)\mapsto f(z,w) -\frac{1}{2\lambda_n}\norm{z-x_n}_{H_1}^2
+\frac{1}{2\lambda_n}\norm{w-y_n}_{H_2}^2
 \end{align}
on $X\times Y$ for all $n\in \N$, 
where $\{\lambda_n\}$ a sequence of positive real numbers such that 
 $\inf_{n}\lambda_n >0$. Then 
$\{(x_n,y_n)\}$ is bounded if and only if $\Saddle (f)$ 
is nonempty. Further, if $\Saddle (f)$ is nonempty, then 
$\{(x_n,y_n)\}$ is weakly convergent to an element of 
$\Saddle (f)$. 
\end{theorem}

Later, applying convergence theorems 
obtained in~\cites{MR2058504, MR2112848}, 
Kohsaka and Takahashi~\cite{MR2144044} obtained 
weak and strong convergence theorems for 
minimax problems in Banach spaces; 
see also~\cites{MR3013139, MR2548424}.  
Aoyama, Kimura, and Takahashi~\cite{MR2422998} 
studied the relation between 
the zero point problem for maximal monotone operators and 
the equilibrium problem for two variable functions in Banach spaces. 

In recent years, the problems in convex analysis and fixed point theory 
have been studied in the more general nonlinear settings 
such as geodesic metric spaces including complete $\CAT(\kappa)$ spaces, where 
$\kappa$ is a real number. This space can be seen as 
both a nonlinear generalization of Hilbert spaces 
and a nonsmooth generalization of Riemannian manifolds 
with bounded sectional curvature. 
In particular, using results on resolvents of convex functions 
obtained in~\cites{MR1360608, MR1651416}, 
Ba{\v{c}}{\'a}k~\cite{MR3047087} obtained a $\Delta$-convergence 
theorem for this algorithm for proper lower semicontinuous convex 
functions in Hadamard spaces, i.e., complete $\CAT(0)$ spaces. 
More recently, Kimura and Kohsaka~\cite{MR3574140} studied two 
modifications of the algorithm for such functions in Hadamard spaces. 

Assuming the convex hull finite property 
on Hadamard spaces, Kimura and Kishi~\cite{MR3897196} 
studied equilibrium problems in Hadamard spaces. 
Applying KKM theorem in Hadamard spaces obtained by 
Niculescu and Roven\c ta~\cite{MR2561730}, 
Kimura and Kishi~\cite{MR3897196} discussed fundamental properties of 
resolvents of bifunctions in Hadamard spaces. 
Note that a $\CAT(0)$ space $E$ has the convex hull finite 
property if each continuous mapping 
of $\overline{\Co}\, F$ into itself 
has a fixed point whenever $F$  
is a nonempty finite subset of $E$. 
See also Kumam and Chaipunya~\cite{MR4049738} 
on the study of the equilibrium problem in Hadamard spaces without 
the convex hull finite property. 
More recently, Kimura~\cite{MR4286979} 
also obtained results for solving the equilibrium problem 
in complete $\CAT(1)$ spaces with the convex hull finite property. 

This paper is organized as follows: 
In Section~\ref{sec:preliminaries}, we give 
necessary definitions and recall some results needed in this paper. 
In Section~\ref{sec:Sion}, we obtain an analogous result of 
Theorem~\ref{thm:Sion-minimax} in Hadamard spaces 
without assuming either the convex hull finite property or 
the metric completeness. 
In Section~\ref{sec:coercive}, we obtain saddle point theorem 
for coercive saddle functions in Hadamard spaces. 
In Section~\ref{sec:resolvent}, we give the definition of 
resolvents of saddle functions and discuss their fundamental properties. 
In Section~\ref{sec:ppa}, we apply our results in this paper 
to the proximal point algorithm for saddle functions in Hadamard spaces 
and obtain a generalization of Theorem~\ref{thm:Rockafellar-ppa-saddle} 
in Hadamard spaces. 

\section{Preliminaries}\label{sec:preliminaries}

Let $\N$ and $\R$ be the sets of all positive integers and 
all real numbers, respectively. We denote by $\R^2$ 
the two dimensional Euclidean space with 
norm $\norm{\,\cdot \,}_{\R^2}$. 
We denote by $(-\infty, \infty]$ and $[-\infty, \infty)$ 
the sets $\R \cup \{\infty\}$ and $\R \cup \{-\infty\}$, respectively. 
For a nonempty set $X$ and a mapping $T\colon X\to X$, 
we denote by $\Fix(T)$ the set of all fixed points of $T$, i.e., 
$\Fix(T)$ is the set of all $u\in X$ such that $Tu=u$.  
We denote by $\Argmin_X f$ or $\Argmin_{y\in X}f(y)$ 
the set of all minimizers of a function $f\colon X\to (-\infty, \infty]$, i.e., 
\begin{align}
 \Argmin_X f = \Argmin_{y\in X}f(y) = \bigl\{u\in X: f(u)=\inf f(X)\bigr\}. 
\end{align}
We also denote by $\Argmax_X g$ or $\Argmax_{y\in X}g(y)$ 
the set of all maximizers of a function $g\colon X\to [-\infty, \infty)$, i.e., 
\begin{align}
 \Argmax_X g = \Argmax_{y\in X}g(y) = \bigl\{u\in X: g(u)=\sup g(X)\bigr\}. 
\end{align}
Let $(X,d)$ be a metric space. A mapping $T\colon X\to X$ is said to be:
\begin{itemize}
 \item nonexpansive if $d(Tx,Ty)\leq d(x,y)$ 
for all $x,y\in X$; 
 \item quasi-nonexpansive if 
$\Fix(T)$ is nonempty and $d(Tx,y)\leq d(x,y)$ 
for all $x\in X$ and $y\in \Fix(T)$; 
 \item firmly metrically nonspreading~\cite{MR4021035} if 
\begin{align}
 2d(Tx,Ty)^2 +d(Tx,x)^2 + d(Ty,y)^2 
 \leq d(Tx,y)^2 + d(Ty,x)^2
\end{align}
 for all $x,y\in X$; 
 \item metrically nonspreading~\cite{MR4021035} if 
\begin{align}
 2d(Tx,Ty)^2 \leq d(Tx,y)^2 + d(Ty,x)^2
\end{align}
 for all $x,y\in X$. 
\end{itemize}
It is obvious that if $T$ is firmly metrically nonspreading, then 
it is metrically nonspreading. 

Let $(X,d)$ be a metric space. The space $X$ is said to be 
geodesic if for each $x,y\in X$, there exists 
a mapping $\gamma\colon [0,1]\to X$ such that 
$\gamma(0)=x$, $\gamma(1)=y$, and 
\begin{align}
 d\bigl(\gamma(s), \gamma(t)\bigr) = d(x,y)\abs{s-t}
\end{align}
for all $s,t\in [0,1]$. 
The mapping $\gamma$ is called a normalized geodesic from $x$ to $y$. 
Its image is denoted by $[x,y]_{\gamma}$ 
and is called a geodesic segment between $x$ and $y$.  
The convex combination between $x$ and $y$ is defined by 
\begin{align}
 (1-\alpha) x\oplus_{\gamma} \alpha y = \gamma (\alpha)
\end{align}
for all $\alpha \in [0,1]$. 
The space $X$ is also said to be uniquely geodesic if 
for each $x,y\in X$, there exists a unique 
normalized geodesic $\gamma$ from $x$ to $y$. 
In this case, the geodesic segment and 
the convex combination between $x$ and $y$ are simply denoted by $[x,y]$ 
and $(1-\alpha) x\oplus \alpha y$ for all $\alpha \in [0,1]$, 
respectively. 

A geodesic metric space $(X,d)$ is said to be a $\CAT(0)$ space 
if 
\begin{align}
 \begin{split}
 &d\bigl(
(1-\alpha)x\oplus_{\gamma_1} \alpha y, 
(1-\beta)x \oplus_{\gamma_2} \beta z
 \bigr) \\
 &\quad \leq 
 \norm{
 (1-\alpha)\overline{x}+ \alpha \overline{y}
 - 
 \bigl((1-\beta)\overline{x} + \beta \overline{z}\bigr)
 }_{\R^2}  
 \end{split}
\end{align}
whenever $\alpha, \beta\in [0,1]$, 
$x,y,z\in X$, 
$\gamma_1$ and $\gamma_2$ are normalized geodesics 
from $x$ to $y$ and from $x$ to $z$, respectively, 
$\overline{x}, \overline{y}, \overline{z}\in \R^2$, 
and 
\begin{align}\label{eq:comparison-triangle}
 d(x,y)=\norm{\overline{x}-\overline{y}}_{\R^2}, \quad 
 d(y,z)=\norm{\overline{y}-\overline{z}}_{\R^2}, \quad 
 d(z,x)=\norm{\overline{z}-\overline{x}}_{\R^2}.  
\end{align}
Note that if $x,y,z\in X$, then the triangle inequality on $X$ 
implies that there exist 
$\overline{x}, \overline{y}, \overline{z}\in \R^2$ 
such that~\eqref{eq:comparison-triangle} holds. 

In fact, we may assume that $d(x,y)\geq d(y, z)$ and 
$d(x,y)\geq d(z,x)$ and set 
\begin{align}
 a=d(y,z), \quad b=d(z,x), \quad c=d(x,y). 
\end{align}
If $c=0$, then $a=b=c=0$ and hence it suffices to set 
$\overline{x}=\overline{y}=\overline{z}=(0,0)$. 
If $c>0$, then it suffices to 
set 
\begin{align}
 \begin{split}
 \overline{x}&=(0,0), \quad 
 \overline{y}=\bigl(c,0\bigr), \\
 \overline{z}&= 
 \left(\frac{b^2+c^2-a^2}{2c}, 
 \frac{\sqrt{(a+b+c)(a+b-c)(b+c-a)(c+a-b)}}{2c}
 \right),   
 \end{split}
\end{align}
where $\overline{z}=(s,t)$ is a solution to the simultaneous equation 
\begin{align}
 s^2+t^2=b^2, \quad (s-c)^2 + t^2 = a^2. 
\end{align}
It is clear that if $(X,d)$ is a $\CAT(0)$ space, then 
it is uniquely geodesic. 

Every complete $\CAT(0)$ space is called an Hadamard space. 
It is known that every nonempty closed convex subset of a real Hilbert space, 
every complete $\R$-tree, 
and every complete simply connected Riemannian manifold 
with nonpositive sectional curvature are 
Hadamard spaces. 
See~\cites{MR3241330, MR1744486, MR1835418} for more details 
on geodesic metric spaces and $\CAT(0)$ spaces. 

If $(X,d)$ is a $\CAT(0)$ space, then 
the following fundamental inequalities 
\begin{align}\label{eq:convexity-d}
 d\bigl((1-\alpha)x\oplus \alpha y, z\bigr) 
 \leq (1-\alpha)d(x,z) + \alpha d(y,z)
\end{align}
and 
\begin{align}\label{eq:CN-inequality}
 d\bigl((1-\alpha)x\oplus \alpha y, z\bigr)^2
 \leq (1-\alpha)d(x,z)^2 + \alpha d(y,z)^2 -\alpha(1-\alpha)d(x,y)^2
\end{align}
hold for all $x,y,z\in X$ and $\alpha\in [0,1]$. 

We know the following results: 

\begin{theorem}[{\cite[Theorem~1.3.3]{MR3241330}}]\label{thm:CS-inequality}
Let $(X,d)$ be a geodesic metric space. 
Then $X$ is a $\CAT(0)$ space if and only if the Cauchy--Schwarz inequality
\begin{align}
  \frac{1}{2}
 \left\{d(x_1,x_4)^2+d(x_2,x_3)^2-d(x_1,x_3)^2-d(x_2,x_4)^2\right\}
 \leq d(x_1,x_2)d(x_3,x_4)
\end{align}
holds for all $x_1,x_2,x_3,x_4\in X$. 
\end{theorem}

\begin{theorem}[\cite{MR4021035}]\label{thm:metrically-nonspareding}
 Let $X$ be an Hadamard space and $T\colon X\to X$ a metrically nonspreading mapping. 
 Then $\Fix(T)$ is nonempty if and only if $\{T^nx\}$ is bounded for some $x\in X$. 
 In this case, $\{T^nx\}$ is $\Delta$-convergent to an element of $\Fix(T)$ 
 for each $x\in X$. 
\end{theorem}

As in the proof of~\cite[Lemma~2.5]{MR3777000}, we can prove the following: 

\begin{lemma}\label{lem:f-net}
 Let $I$ be a nonempty closed subset of $\R$, 
 $\{t_{\alpha}\}_{\alpha \in A}$ a net in $I$ such that 
 $\{t_{\alpha}\}_{\alpha\geq \alpha_0}$ is bounded 
 for some $\alpha_0\in A$, and 
 $f\colon I\to \R$ a nondecreasing and continuous function. 
 Then 
 \begin{align*}
  f\left(\limsup_{\alpha}t_{\alpha}\right) 
  = \limsup_{\alpha} f(t_{\alpha}). 
 \end{align*}
\end{lemma}

If $C$ is a nonempty closed convex subset of an Hadamard space 
and $x\in X$, then there exists a unique $\hat{x}\in C$ such that 
\begin{align}
 d(\hat{x}, x) = \inf_{y\in C} d(y,x). 
\end{align}
The metric projection $P_C$ of $X$ onto $C$ is defined by 
$P_C(x)=\hat{x}$ for all $x\in X$. 
It is well known that $P_C$ is firmly metrically nonspreading. 
This implies that if $x\in X$ and $y\in C$, then 
\begin{align}
 d\bigl(P_C(x), y\bigr)^2 + d\bigl(P_C(x),x\bigr)^2 
 \leq d(x, y)^2. 
\end{align}
It also follows from the firmly metrical nonspreadingness of $P_C$ 
and Theorem~\ref{thm:CS-inequality} that 
$P_C$ is nonexpansive. 
In particular, $P_C$ is quasi-nonexpansive. 

Let $(X,d)$ be a $\CAT(0)$ space and 
$\{x_n\}$ a sequence in $X$. Then 
the asymptotic center $\AC\bigl(\{x_n\}\bigr)$ of 
$\{x_n\}$ is defined by 
\begin{align}
 \AC\bigl(\{x_n\}\bigr) 
 = \left\{
 u\in X: \limsup_{n\to \infty}d(u, x_n) = \inf_{y\in X} 
 \limsup_{n\to \infty} d(y, x_n)
 \right\}. 
\end{align}
It is obvious that $\AC\bigl(\{x_n\}\bigr)=X$ 
if $\{x_n\}$ is unbounded. 
The sequence $\{x_n\}$ is said to be $\Delta$-convergent to $p\in X$ 
if $\AC\bigl(\{x_{n_i}\}\bigr) =\{p\}$ 
for each subsequence $\{x_{n_i}\}$ of $\{x_n\}$. 
In this case, $\{x_n\}$ is bounded. 
If $\{x_n\}$ is a sequence in 
a nonempty closed convex subset $X$ of a real Hilbert space 
and $p\in X$, then 
$\{x_n\}$ is $\Delta$-convergent to $p$ 
if and only if it is weakly convergent to $p$. 
We denote by $\omega_{\Delta}\bigl(\{x_n\}\bigr)$ 
the set of all $p\in X$ such that 
there exists a subsequence $\{x_{n_i}\}$ of $\{x_n\}$ 
such that $\{x_{n_i}\}$ is $\Delta$-convergent to $p$. 
For a net $\{x_{\alpha}\}$ in $X$, we can 
similarly define its asymptotic center by 
\begin{align}
 \AC\bigl(\{x_{\alpha}\}\bigr) 
 = \left\{
 u\in X: \limsup_{\alpha}d(u, x_{\alpha}) 
 = \inf_{y\in X} 
 \limsup_{\alpha} d(y, x_{\alpha})
 \right\}. 
\end{align}
It is said to be $\Delta$-convergent to $p\in X$ if 
$\AC\bigl(\{x_{\alpha_{\beta}}\}\bigr)=\{p\}$ 
for each subnet $\{x_{\alpha_{\beta}}\}$ of 
$\{x_{\alpha}\}$. 
For a sequence $\{x_n\}$ in $X$ and $p\in X$, 
it seems not to be clear whether 
$\{x_n\}$ is $\Delta$-convergent to $p$ 
as a sequence if and only if 
$\{x_n\}$ is $\Delta$-convergent to $p$ 
as a net. The following lemma answers to this question 
affirmatively: 

\begin{lemma}\label{lem:Delta-net}
 Let $(X,d)$ be a $\CAT(0)$ space and $\bar{x}$ an element in $X$. 
 Then the following hold: 
 \begin{enumerate}
  \item[(i)] If $\{x_{\alpha}\}_{\alpha\in A}$ is a net in $X$, 
 then $\{x_{\alpha}\}$ is $\Delta$-convergent to $\bar{x}$ 
 if and only if 
 $\{x_{\alpha}\}_{\alpha \geq \alpha_0}$ is bounded 
 for some $\alpha_0\in A$ and 
 $\{P_{[\bar{x},y]}x_{\alpha}\}$ is convergent to $\bar{x}$ 
 for all $y\in X$; 
  \item[(ii)] if $\{x_{n}\}$ is a sequence in $X$, 
 then $\{x_{n}\}$ is $\Delta$-convergent to $\bar{x}$ 
 as a sequence  
 if and only if 
 $\{x_{n}\}$ is bounded and 
 $\{P_{[\bar{x},y]}x_{n}\}$ is convergent to $\bar{x}$ 
 for all $y\in X$; 
  \item[(iii)] if $\{x_{n}\}$ is a sequence in $X$, 
 then $\{x_{n}\}$ is $\Delta$-convergent to $\bar{x}$ 
 as a sequence  
 if and only if 
 it is $\Delta$-convergent to $\bar{x}$ 
 as a net. 
 \end{enumerate}
\end{lemma}

\begin{proof}
 Note that if $u,v\in X$, then 
 $[u,v]$ is a nonempty compact convex subset of $X$. 
 Thus the metric projection $P_{[u,v]}$ of $X$ onto $[u,v]$ 
 is well defined and 
 satisfies the following inequality: 
 \begin{align}
  d\bigl(z, P_{[u,v]}x\bigr)^2 + d\bigl(P_{[u,v]}x, x\bigr)^2 
  \leq d(z, x)^2
 \end{align}
 for all $z\in [u,v]$ and $x\in X$. 
 
 We first prove the if part of~(i). 
 Suppose that $\{x_{\alpha}\}_{\alpha \geq \alpha_0}$ is bounded 
 for some $\alpha_0\in A$ and 
 $\{P_{[\bar{x},y]}x_{\alpha}\}$ is convergent to $\bar{x}$ 
 for all $y\in X$. Let $y\in X\setminus \{\bar{x}\}$ be given 
 and let $\{x_{\alpha_{\beta}}\}_{\beta\in B}$ be a subnet of $\{x_{\alpha}\}$. 
 Note that there exists $\beta_0\in B$ such that 
 $\beta_0\leq \beta$ implies that $\alpha_0\leq \alpha_{\beta}$. 
 Thus $\{x_{\alpha_{\beta}}\}_{\beta\geq \beta_0}$ is bounded. 
 This implies that all the upper limits in the proof of 
 this part are finite. 
 Set $P=P_{[\bar{x},y]}$. 
 Since 
 \begin{align}
 \begin{split}
  &\Bigl\lvert 
  d(Px_{\alpha_{\beta}},x_{\alpha_{\beta}})^2 - d(\bar{x},x_{\alpha_{\beta}})^2
  \Bigr\lvert \\
  &\quad \leq 
  \bigl\{
   d(Px_{\alpha_{\beta}},x_{\alpha_{\beta}}) + d(\bar{x},x_{\alpha_{\beta}})
  \bigr\}
    d(Px_{\alpha_{\beta}},\bar{x})  \to 0, 
 \end{split}
 \end{align}
 we have 
 \begin{align}
  \limsup_{\beta}d(Px_{\alpha_{\beta}},x_{\alpha_{\beta}})^2 
  =\limsup_{\beta}d(\bar{x},x_{\alpha_{\beta}})^2. 
 \end{align}
 Then we have 
 \begin{align}
  \begin{split}
   \limsup_{\beta}d(y,x_{\alpha_{\beta}})^2 
   &\geq    \limsup_{\beta}
  \bigl\{
   d(y,Px_{\alpha_{\beta}})^2 + d(Px_{\alpha_{\beta}},x_{\alpha_{\beta}})^2 
  \bigr\} \\
   &=    d(y,\bar{x})^2 
 +\limsup_{\beta}
  d(Px_{\alpha_{\beta}},x_{\alpha_{\beta}})^2 \\
   &=    d(y,\bar{x})^2 
 +\limsup_{\beta}
  d(\bar{x},x_{\alpha_{\beta}})^2 
   >
 \limsup_{\beta}
  d(\bar{x},x_{\alpha_{\beta}})^2. 
  \end{split}
 \end{align}
Using Lemma~\ref{lem:f-net}, we obtain 
\begin{align}
 \begin{split}
\Bigl\{\limsup_{\beta}d(y,x_{\alpha_{\beta}})\Bigr\}^2 
 &=\limsup_{\beta}d(y,x_{\alpha_{\beta}})^2 \\
 &> \limsup_{\beta}d(\bar{x},x_{\alpha_{\beta}})^2
 = \Bigl\{\limsup_{\beta}d(\bar{x},x_{\alpha_{\beta}})\Bigr\}^2. 
 \end{split}
\end{align}
 This means that 
 $\AC\bigl(\{x_{\alpha_{\beta}}\}\bigr)=\{\bar{x}\}$ 
 and hence $\{x_{\alpha}\}$ is $\Delta$-convergent to $\bar{x}$. 

 We next prove the only if part of~(i).  
 Suppose that 
 $\{x_{\alpha}\}_{\alpha\in A}$ is $\Delta$-convergent to $\bar{x}$. 
 We prove that $\{x_{\alpha}\}_{\alpha\geq \alpha_0}$ 
 is bounded for some $\alpha_0 \in A$. 
 We may assume that $X\neq \{\bar{x}\}$. 
 Then we have $y\in X\setminus \{\bar{x}\}$. 
 By assumption, we have 
 \begin{align}
  \limsup_{\alpha}d(\bar{x},x_{\alpha}) < 
\limsup_{\alpha}d(y,x_{\alpha}). 
 \end{align}
 Thus we have $\alpha_0\in A$ such that 
 \begin{align}
  \sup_{\alpha_0\leq \gamma} d(\bar{x},x_{\gamma}) 
 < \limsup_{\alpha}d(y,x_{\alpha}). 
 \end{align}
 This implies that $\{x_{\alpha}\}_{\alpha\geq \alpha_0}$ is bounded. 
 Suppose that 
 there exists $y\in X$ such that 
 $\{P_{[\bar{x},y]}x_{\alpha}\}$ is not convergent to $\bar{x}$. 
 Set $P=P_{[\bar{x},y]}$. Since $[\bar{x},y]$ is compact, 
 we have a subnet $\{Px_{\alpha_{\beta}}\}$ of $\{Px_{\alpha}\}$ 
 which is convergent to $\hat{x}$. Then we have $\bar{x}\neq \hat{x}$. 
 Since $\AC\bigl(\{x_{\alpha_{\beta}}\}\bigr)=\{\bar{x}\}$, we have  
 and 
 \begin{align}
  \limsup_{\beta} d(\bar{x},x_{\alpha_{\beta}})
  < \limsup_{\beta} d(\hat{x},x_{\alpha_{\beta}}). 
 \end{align}
 On the other hand, since $Px_{\alpha_{\beta}}\to \hat{x}$, we have 
 \begin{align}
  \limsup_{\beta} d(\hat{x},x_{\alpha_{\beta}}) 
  =\limsup_{\beta} d(Px_{\alpha_{\beta}},x_{\alpha_{\beta}}) 
  \leq \limsup_{\beta} d(\bar{x},x_{\alpha_{\beta}}) 
 \end{align}
 This is a contradiction. Thus 
 $\{P_{[\bar{x},y]}x_{\alpha}\}$ is convergent to $\bar{x}$ 
 for all $y\in X$. 

 We can similarly prove the part~(ii). 
 The part~(iii) obviously follows from the previous parts. 
\end{proof}

We know the following results: 

\begin{theorem}[\cite{MR2416076}]\label{thm:Kirk--Panyanak}
 If $\{x_n\}$ is a bounded sequence in an Hadamard space $X$, 
 then there exists $p\in X$ such that $\AC\bigl(\{x_n\}\bigr)=\{p\}$ 
 and there exists a subsequence $\{x_{n_i}\}$ of $\{x_n\}$ 
 which is $\Delta$-convergent to $p$. 
\end{theorem}

\begin{lemma}[\cite{MR3574140}]\label{lem:Kimura-K-omega}
 Let $X$ be an Hadamard space and $\{x_n\}$ a bounded sequence in $X$ such that 
 $\{d(u, x_n)\}$ is convergent 
 for each $u$ in $\omega_{\Delta}\bigl(\{x_n\}\bigr)$. 
 Then $\{x_n\}$ is $\Delta$-convergent to an element of $X$. 
\end{lemma}

A subset $C$ of a $\CAT(0)$ space $X$ is said to be: 
\begin{itemize}
 \item convex if $[x,y]\subset C$ whenever $x,y\in C$; 
 \item $\Delta$-closed if $p\in C$ whenever 
$\{x_{\alpha}\}$ is a net in $C$ which is $\Delta$-convergent to $p\in X$.  
\end{itemize}
We know that if $X$ is an Hadamard space and 
$C$ is a closed and convex subset of $X$, then 
$C$ is $\Delta$-closed. In fact, if $\{x_{\alpha}\}$ is a 
net in $C$ which is $\Delta$-convergent to $\bar{x}\in X$, 
then $d\bigl(P_C(\bar{x}), x_{\alpha}\bigr)
 \leq d\bigl(\bar{x}, x_{\alpha}\bigr)$ for each $\alpha$ 
and hence 
\begin{align*}
 \limsup_{\alpha} d\bigl(P_C(\bar{x}), x_{\alpha}\bigr)
 \leq \limsup_{\alpha}d\bigl(\bar{x}, x_{\alpha}\bigr). 
\end{align*}
Since $\AC\bigl(\{x_{\alpha}\}\bigr)=\{\bar{x}\}$, we have 
$\bar{x}=P_C(\bar{x})\in C$. Thus $C$ is $\Delta$-closed. 

Let $(X,d)$ be a $\CAT(0)$ space, 
$f\colon X\to (-\infty, \infty]$ a function, and 
$C_{\lambda}$ the set defined by 
$C_{\lambda}=\{x\in X: f(x)\leq \lambda\}$ 
for all $\lambda\in \R$. Then $f$ is said to be: 
\begin{itemize}
 \item proper if $f(x)\in \R$ for some $x\in X$; 
 \item convex if 
 \begin{align}
 f\bigl((1-\alpha)x\oplus \alpha y\bigr) 
 \leq (1-\alpha)f(x) + \alpha f(y) 
 \end{align}
 for all $x,y\in X$ and $\alpha \in (0,1)$; 
 \item quasi-convex if $C_{\lambda}$ is convex 
for all $\lambda \in \R$, or equivalently,  
\begin{align}
 f\bigl((1-\alpha)x\oplus \alpha y\bigr) 
 \leq \max\{f(x), f(y)\}
\end{align}
 for all $x,y\in X$ and $\alpha \in (0,1)$; 
 \item lower semicontinuous if $C_{\lambda}$ 
 is closed in $X$ for all $\lambda \in \R$; 
 \item $\Delta$-lower semicontinuous if 
 $C_{\lambda}$ is $\Delta$-closed in $X$ for all $\lambda \in \R$. 
\end{itemize}
It follows from~\eqref{eq:convexity-d} 
and~\eqref{eq:CN-inequality} that 
if $X$ is a $\CAT(0)$ space, then 
$d(\,\cdot\,,z)$ and $d(\,\cdot\,, z)^2$ 
are continuous and convex real functions on $X$ for all $z\in X$. 
A function $g\colon X\to [-\infty, \infty)$ is also said to be: 
proper if $-g$ is proper; concave if $-g$ is convex; 
quasi-concave if $-g$ is quasi-convex; 
upper semicontinuous if $-g$ is lower semicontinuous; 
$\Delta$-upper semicontinuous if 
 $-g$ is $\Delta$-lower semicontinuous. 

If $X$ be an Hadamard space 
and $f\colon X\to (-\infty, \infty]$ is 
 proper, lower semicontinuous, and quasi-convex, 
 then $f$ is $\Delta$-lower semicontinuous. 
Similarly, if $g\colon X\to [-\infty, \infty)$ is 
 proper, upper semicontinuous, and quasi-concave, 
 then $g$ is $\Delta$-upper semicontinuous. 

The following theorems are important: 

\begin{theorem}[{\cite[Lemma~2.2.13]{MR3241330}}]\label{thm:Bacak-umbrella}
 Let $(X,d)$ be an Hadamard space, 
 $f\colon X\to (-\infty, \infty]$ a proper lower semicontinuous 
 convex function, and $p\in X$. Then 
 there exists a positive real number $C$ such that 
 \begin{align}
  f(x) \geq -C \bigl(d(x, p) + 1\bigr)
 \end{align}
 for all $x\in X$. 
\end{theorem}

\begin{theorem}[\cite{MR3574140}]\label{thm:Kimura-K-bdd}
 Let $(X,d)$ be an Hadamard space, 
 $\{x_n\}$ a bounded sequence in $X$, 
 $\{\beta_n\}$ a sequence of positive real numbers such that 
 $\sum_{n=1}^{\infty}\beta_n =\infty$, 
 and $g\colon X\to \R$ the function defined by 
 \begin{align}
  g(y)=\limsup_{n\to \infty}\frac{1}{\sum_{k=1}^n\beta_k}
  \sum_{l=1}^n \beta_l d(y,x_l)^2
 \end{align}
 for all $y\in X$. Then $g$ is a continuous and convex function such that 
 $\Argmin_X g$ is a singleton. 
\end{theorem}

We know the following lemma: 
\begin{lemma}
Let $X$ be a $\CAT(0)$ space and $f\colon X\to (-\infty, \infty]$ 
a function. 
Then $f$ is $\Delta$-lower semicontinuous if and only if 
\begin{align}
 f(u)\leq \liminf_{\alpha}f(x_{\alpha})
\end{align}
whenever 
$\{x_{\alpha}\}$ is a net in $X$ which is $\Delta$-convergent 
to $u\in X$.  
Similarly, $g\colon X\to [-\infty, \infty)$ 
is $\Delta$-upper semicontinuous if and only if 
\begin{align}
 g(u)\geq \limsup_{\alpha}g(x_{\alpha})
\end{align}
whenever $\{x_{\alpha}\}$ is a net in $X$ which is $\Delta$-convergent 
to $u\in X$. 
\end{lemma}

For the sake of completeness, we give the proof of the first assertion. 
\begin{proof}

Suppose that $f\colon X\to (-\infty, \infty]$ 
is $\Delta$-lower semicontinuous 
and $\{x_{\alpha}\}_{\alpha\in A}$ a net in $X$ which is $\Delta$-convergent 
to $u\in X$. Let $L$ be any real number such that 
$L < f(u)$. Set $C=\{x\in X: f(x)\leq L\}$. 
We claim that there exists 
$\alpha_0\in A$ such that $x_{\alpha}\in X\setminus C$ for all $\alpha \geq \alpha_0$. 
If not, for each $\alpha\in A$, there exists 
$\beta_{\alpha}\in A$ such that $\beta_{\alpha}\geq \alpha$ 
and $x_{\beta_{\alpha}}\in C$. Then $\{x_{\beta_{\alpha}}\}_{\alpha\in A}$ 
is a subnet of $\{x_{\alpha}\}$. Since $\{x_{\alpha}\}$ is $\Delta$-convergent to $u$, 
we know that $\{x_{\beta_{\alpha}}\}_{\alpha\in A}$ is $\Delta$-convergent to $u$. 
Since $C$ is $\Delta$-closed, we have $u\in C$. This is a contradiction. 
Hence we have $\alpha_0\in A$ such that 
$L < f(x_{\alpha})$ for all $\alpha \geq \alpha_0$. This implies that 
\begin{align}
 L \leq \inf_{\alpha\geq \alpha_0} f(x_{\alpha}) 
 \leq \liminf_{\alpha}f(x_{\alpha}). 
\end{align}
Letting $L \to f(u)$, we obtain $f(u)\leq \liminf_{\alpha}f(x_{\alpha})$. 
Conversely, suppose that 
\begin{align}
 f(u)\leq \liminf_{\alpha}f(x_{\alpha})
\end{align}
whenever $\{x_{\alpha}\}_{\alpha \in A}$ is a net in $X$ which is $\Delta$-convergent 
to $u\in X$ and let $\lambda\in \R$ be given. We prove that 
the set $C$ given by $C=\{x\in X: f(x)\leq \lambda\}$ is $\Delta$-closed. 
Let $\{x_{\alpha}\}$ be a net in $C$ which is $\Delta$-convergent to $u\in X$. 
Then we have $f(x_{\alpha})\leq \lambda$ for all $\alpha\in A$. 
Taking the lower limit, we obtain 
\begin{align}
 f(u)\leq \liminf_{\alpha}f(x_{\alpha}) \leq \lambda. 
\end{align}
Thus $u\in C$ and $f$ is $\Delta$-lower semicontinuous.  
\end{proof}

We also know the following lemma: 

\begin{lemma}\label{lem:Delta-bdd}
 Let $X$ be a bounded Hadamard space. Then the following hold: 
 \begin{enumerate}
  \item If $f\colon X\to (-\infty, \infty]$ is $\Delta$-lower semicontinuous, then 
 the set $\Argmin_X f$ is nonempty; 
  \item if $g\colon X\to [-\infty, \infty)$ is $\Delta$-upper semicontinuous, then 
 the set $\Argmax_X g$ is nonempty. 
 \end{enumerate}
\end{lemma}

For the sake of completeness, we give the proof. 

\begin{proof}
 Suppose that $f\colon X\to (-\infty, \infty]$ is $\Delta$-lower semicontinuous. 
 If $f$ is not proper, then $f(x)=\infty$ for all $x\in X$ and hence 
 $\Argmin_X f=X$. Suppose that $f$ is proper and 
 set $L=\inf f(X)$. We know that $L\in [-\infty, \infty)$. 
 Then there exists a sequence $\{x_n\}$ in $X$ such that 
 $f(x_n)\to L$. Since $X$ is bounded, the sequence $\{x_n\}$ 
 is bounded. Theorem~\ref{thm:Kirk--Panyanak} ensures that 
 there exists a subsequence $\{x_{n_i}\}$ of $\{x_n\}$ 
 which is $\Delta$-convergent to some point $u\in X$. 
 Since $f$ is $\Delta$-lower semicontinuous, we obtain 
 \begin{align}
  f(u)\leq \liminf_{i\to \infty}f(x_{n_i}) 
 =\lim_{n\to \infty}f(x_n) = L. 
 \end{align}
 Hence $f(u)=\inf f(X)$. Therefore $\Argmin_X f$ is nonempty. 
\end{proof}

Kimura and Kishi~\cite{MR3897196} obtained the following 
fundamental lemma: 

\begin{lemma}[{\cite[Lemma~3.3]{MR3897196}}]\label{lem:Delta-compact}
 Let $X$ be a bounded Hadamard space and  
 $\{C_{\lambda}\}_{\lambda\in \Lambda}$ a family of 
 $\Delta$-closed subsets of $X$ 
 such that 
 \begin{align}
  \bigcap_{k=1}^m C_{\lambda_k}\neq \emptyset
 \end{align}
 whenever $m\in \N$ and $\lambda_1,\lambda_2,\dots ,\lambda_m\in \Lambda$. 
 Then $\bigcap_{\lambda\in \Lambda}C_{\lambda}$ is nonempty. 
\end{lemma}

Let $(X,d_X)$ and $(Y,d_Y)$ be $\CAT(0)$ spaces. 
Then the $\ell^2$-metric $d$ on $X\times Y$ is defined by 
\begin{align}
 d\bigl((x,y),(x',y')\bigr) = 
 \sqrt{d_X(x,x')^2 + d_Y(y,y')^2} 
\end{align}
for all $(x,y)$ and $(x',y')$ in $X\times Y$. 
If $u=(x,y)$ and $v=(x',y')$ are points in $X\times Y$, 
then a normalized geodesic $\gamma$ from $u$ to $v$ 
is given by 
\begin{align}\label{eq:geodesic-product}
 \gamma(t) = \bigl((1-t)x\oplus t x', (1-t)y \oplus t y'\bigr)
\end{align}
for all $t\in [0,1]$. 
It is clear that $\gamma(0)=u$ and $\gamma(1)=v$. 
Further, if $s, t\in [0,1]$, then 
\begin{align}
 \begin{split}
 d\bigl(\gamma(s), \gamma(t)\bigr) 
 &=\sqrt{d_X(x,x')^2(s-t)^2+d_Y(y,y')^2(s-t)^2} 
 =d(u,v) \abs{s-t}.   
 \end{split}
\end{align}
Thus $X\times Y$ is a geodesic space. 
If $u_i=(x_i,y_i) \in X\times Y$ for all $i=1,2,3,4$, 
then Theorem~\ref{thm:CS-inequality} and 
the Cauchy--Shwarz inequality on $\R^2$ imply that 
\begin{align}
 \begin{split}
 &d(u_1,u_4)^2 + d(u_2,u_3)^2 -d(u_1,u_3)^2 -d(u_2,u_4)^2 \\
 &= d_X(x_1,x_4)^2 + d_X(x_2,x_3)^2 -d_X(x_1,x_3)^2 -d_X(x_2,x_4)^2 \\
 &\quad + d_Y(y_1,y_4)^2 + d_Y(y_2,y_3)^2 -d_Y(y_1,y_3)^2 -d_Y(y_2,y_4)^2 \\
 &\leq 2\,d_X(x_1,x_2)d_X(x_3,x_4) + 2\,d_Y(y_1,y_2)d_Y(y_3,y_4) \\
 &\leq 2\,d(u_1,u_2)d(u_3,u_4).   
 \end{split}
\end{align} 
Thus it also follows from Theorem~\ref{thm:CS-inequality} that 
$X\times Y$ is a $\CAT(0)$ space. 

\section{Sion's minimax theorem in Hadamard spaces}\label{sec:Sion}

In this section, using the methods by Komiya~\cite{MR0930413}, we give the proof of 
Sion's minimax theorem in Hadamard spaces. 

\begin{theorem}\label{thm:Sion-CAT0}
 Let $(X,d_X)$ an Hadamard space, $(Y,d_Y)$ an bounded Hadamard space, and 
 $f\colon X\times Y\to \R$ a function satisfying 
 \begin{enumerate}
  \item[(i)] $f(\,\cdot\,,y)$ is upper semicontinuous 
 and quasi-concave for each $y\in Y$; 
  \item[(ii)] $f(x,\cdot \,)$ is lower semicontinuous 
 and quasi-convex for each $x\in X$. 
 \end{enumerate}
 Then the equality 
 \begin{align}
  \sup_{x\in X}\min_{y\in Y} f(x,y)=\min_{y\in Y} \sup_{x\in X} f(x,y) 
 \end{align}
 holds. 
\end{theorem}

\begin{remark}\label{rem:Sion-CAT0}
 For each $x\in X$ and $y\in Y$, the functions 
 $f(\,\cdot\,, y)$ and $f(x, \cdot\,)$ are $\Delta$-upper semicontinuous 
 and $\Delta$-lower semicontinuous, respectively. 
 Since $Y$ is a bounded Hadamard space, 
 it follows from Lemma~\ref{lem:Delta-bdd} that 
 $f(x,\,\cdot\,)$ has a minimizer. Thus $\min_{y\in Y}f(x,y)\in \R$ 
 for all $x\in X$. This implies that 
 \begin{align}
  \sup_{x\in X}\min_{y\in Y}f(x,y)\in (-\infty, \infty]. 
 \end{align}
 On the other hand, the function 
 $y\mapsto \sup_{x\in X}f(x,y)$ 
 is a $\Delta$-lower semicontinuous function of $Y$ into $(-\infty, \infty]$. 
 Then it follows from Lemma~\ref{lem:Delta-bdd} that 
 \begin{align}
  \min_{y\in Y}\sup_{x\in X} f(x,y)\in (-\infty, \infty]. 
 \end{align}
 Further, since $f(x,y)\leq \sup_{x\in X}f(x,y)$ 
 for all $x\in X$ and $y\in Y$, we have 
 \begin{align}
  \min_{y\in Y}f(x,y)\leq \min_{y\in Y}\sup_{x\in X}f(x,y)
 \end{align}
 for all $x\in X$ and hence 
 \begin{align}
  \sup_{x\in X}\min_{y\in Y}f(x,y)\leq \min_{y\in Y}\sup_{x\in X}f(x,y). 
 \end{align}
 Thus we need to prove the reverse inequality. 
\end{remark}

\begin{lemma}\label{lem:Sion-CAT0-1}
 Let $X$, $Y$, and $f$ be the same as in Theorem~\ref{thm:Sion-CAT0}. 
 Suppose that $x_1,x_2\in X$ and $\alpha \in \R$ satisfy 
 \begin{align}
  \alpha < \min_{y\in Y}\max\{f(x_1, y), f(x_2, y)\}. 
 \end{align}
 Then there exists $x_0\in X$ such that 
 $\alpha < \min_{y\in Y}f(x_0,y)$. 
\end{lemma}

\begin{proof}
 The proof is by contradiction. Suppose that 
 $\alpha \geq \min_{y\in Y}f(x, y)$ 
 for all $x\in X$. 
 By assumption, there exists a real number $\beta$ such that 
 \begin{align}
  \alpha < \beta < \min_{y\in Y}\max\{f(x_1, y), f(x_2, y)\}. 
 \end{align}
 Then we define two sets $C_z$ and $C'_z$ by 
 \begin{align}
  C_z = \{y\in Y: f(z,y) \leq \alpha\} 
  \quad \textrm{and} \quad 
  C'_z = \{y\in Y: f(z,y) \leq \beta\} 
 \end{align}
 for all $z\in [x_1,x_2]$. We denote by $A$ and $B$ the sets given by 
 \begin{align}
  A=C'_{x_1} \quad \textrm{and}\quad B=C'_{x_2}. 
 \end{align}
 It is clear that $C_z\subset C'_z$ for all $z\in [x_1,x_2]$. 
 By assumption, we know that $C_z$, $C'_z$, $A$, and $B$ are 
 nonempty closed convex subsets of $Y$. 
 We know that $A\cap B$ is empty. 
 If not, we have $y_0\in A\cap B$. Then 
\begin{align}
 f(x_1,y_0)\leq \beta \quad \textrm{and} \quad 
 f(x_2,y_0) \leq \beta
\end{align}
 and hence 
 \begin{align}
  \beta 
& \geq \max \{f(x_1,y_0), f(x_2,y_0)\} 
 \geq \min_{y\in Y}\max \{f(x_1,y), f(x_2,y)\} >\beta. 
 \end{align}
 This is a contradiction. Thus $A\cap B$ is empty. 
 
 Since $f(\,\cdot\,, y)$ is quasi-concave, we have 
 \begin{align}
  f(z,y) \geq \min\{f(x_1,y), f(x_2,y)\}
 \end{align}  
 for all $z\in [x_1,x_2]$ and $y\in Y$. This gives us that 
 $C'_z\subset A\cup B$ for all $z\in [x_1,x_2]$. 
 We next prove the following implication holds true: 
 \begin{center}
 $z\in [x_1,x_2]$ $\Longrightarrow$
 either $C'_z\subset A$ or $C'_z\subset B$ holds. 
 \end{center}
 Let $z\in [x_1,x_2]$ be given. 
 If $z=x_1$, then $C'_z=A$. In this case, 
 $C'_z$ is not a subset of $B$ since $A\cap B$ is empty. 
 Similarly, if $z=x_2$, then $C'_z=B$ and 
 $C'_z$ is not a subset of $A$. 
 So, we may assume that $z\neq x_1$ and $z\neq x_2$. 
 Since $C'_z$ is nonempty, we can fix $p\in C'_z$. 
 Since $C'_z\subset A\cup B$, we have $p\in A\cup B$. 
 We consider the case where 
 $p\in A$ and we prove that $C'_z\subset A$. 
 Let $q\in C'_z$ be given. Then $q\in A\cup B$. 
 We prove that $q\in A$ by contradiction. 
 Assume that $q\notin A$. 
 Then we have $q\in B$. Set 
 \begin{align}
  \lambda_0 = \sup\bigl\{\lambda\in [0,1]: 
 (1-\lambda)p\oplus \lambda q \in A\bigr\}. 
 \end{align}
 Then we have a sequence $\{\lambda_n\}$ in $[0,1]$ such that 
 \begin{align}
  (1-\lambda_n)p \oplus \lambda_n q \in A 
  \quad (\forall n\in \N) \quad \textrm{and} \quad 
  \lim_{n\to \infty}\lambda_n = \lambda_0. 
 \end{align}
 It is obvious that $\lambda_0\in [0,1]$. 
 Since $A$ is closed and 
 \begin{align}
  \lim_{n\to \infty} \bigl(
  (1-\lambda_n)p \oplus \lambda_n q 
  \bigr)
  = (1-\lambda_0)p \oplus \lambda_0 q, 
 \end{align}
 we have 
 $(1-\lambda_0)p \oplus \lambda_0 q \in A$. 
 Thus we obtain 
 \begin{align}
   \lambda_0 = \max\bigl\{\lambda\in [0,1]: 
 (1-\lambda)p\oplus \lambda q \in A\bigr\}. 
 \end{align}
 Since $q\notin A$, we have $\lambda_0\neq 1$. 
 If we take $\mu \in (\lambda_0, 1]$, then 
 $(1-\mu)p \oplus \mu q\notin A$. 
 Since 
 $p, q\in C'_z$, $C'_z$ is convex, and  
 $C'_z\subset A\cup B$, we have 
 \begin{align}
  (1-\mu)p \oplus \mu q \in C'_z \subset A\cup B
 \end{align}
 and hence $(1-\mu)p \oplus \mu q \in B$.  
 Letting $\mu \downarrow \lambda_0$, we have 
 $(1-\lambda_0)p \oplus \lambda_0 q \in B$ 
 since $B$ is closed. 
 This gives us that 
 \begin{align}
  (1-\lambda_0)p \oplus \lambda_0 q \in A\cap B. 
 \end{align}
 This is a contradiction. Consequently, we obtain $q\in A$ 
 and hence $C'_z\subset A$. 
 In this case, 
 $C'_z$ is not a subset of $B$ since $A\cap B$ is empty. 
 Then we also have $C_z\subset C'_z\subset A$. 
 Similarly, in the case where $p\in B$, we can prove that 
 $C'_z\subset B$ and $C'_z$ is not a subset of $A$.
 Then we also have $C_z\subset C'_z\subset B$.  
 Thus either $C'_z\subset A$ or $C'_z\subset B$ 
 holds for all $z\in [x_1,x_2]$. 

 Let $I$ and $J$ be the sets defined by 
 \begin{align}
  \begin{split}
  I&=\bigl\{\lambda\in [0,1]: C_{(1-\lambda)x_1\oplus \lambda x_2} \subset A\bigr\}; \\
  J&=\bigl\{\lambda\in [0,1]: C_{(1-\lambda)x_1\oplus \lambda x_2} \subset B\bigr\}.    
  \end{split}
 \end{align}
 Then we have from $A\cap B=\emptyset$ that $I\cap J$ is empty. 
 Since $C_{x_1}\subset C'_{x_1}=A$, we have $0\in I$. 
 Since $C_{x_2}\subset C'_{x_2}=B$, we have $1\in J$. 
 We next prove that $I$ is closed. 
 Let $\{\lambda_n\}$ be a sequence in $I$ such that 
 $\lambda_n\to \lambda$ and set $z=(1-\lambda)x_1\oplus \lambda x_2$.  
 Let $y\in C_{z}$ be given. 
 Then we know that $f(z,y)\leq \alpha < \beta$. 
 Setting 
 \begin{align}
  z_n=(1-\lambda_n)x_1\oplus \lambda_n x_2
 \end{align}
 for all $n\in \N$, we have from the upper semicontinuity of $f(\,\cdot\,,y)$ that 
 \begin{align}
  \limsup_{n\to \infty}f(z_n, y)\leq f(z, y)<\beta. 
 \end{align}
 Thus there exists $n_0\in \N$ such that $\sup_{k\geq n_0}f(z_k,y)<\beta$. 
 This gives us that $f(z_{n_0},y)<\beta$ 
 and hence $y\in C'_{z_{n_0}}$. 
 On the other hand, since $\lambda_n \in I$, we have 
 \begin{align}
  C_{z_{n_0}}\subset A. 
 \end{align}
 Since $C_{z_{n_0}}\subset C'_{z_{n_0}}$, 
 either $C'_{z_{n_0}}\subset A$ or $C'_{z_{n_0}} \subset B$ holds, 
 and $A\cap B = \emptyset$, 
 we know that $C'_{z_{n_0}}\subset A$. 
 Hence we have $y\in A$. 
 Therefore we have 
 \begin{align}
  C_{(1-\lambda)y_1\oplus \lambda y_2} =C_z \subset A. 
 \end{align}
 This implies that $\lambda\in I$. Thus $I$ is closed. 
 Similarly, we can prove that $J$ is closed. 
 Setting $\lambda^*=\sup I$, we can show that 
 $\lambda^*=\max I$ and $0\leq \lambda^*<1$. 
 If $\mu \in (\lambda^*, 1]$,  
 then $\mu \in J$. Letting $\mu \downarrow \lambda^*$, 
 we obtain $\lambda^* \in J$ since $J$ is closed. Thus we have 
 $\lambda ^*\in I\cap J$. This contradicts $I\cap J=\emptyset$.  
\end{proof}

\begin{lemma}\label{lem:Sion-CAT0-2}
 Let $m\in \N$ be given and $X$, $Y$, and $f$ 
 the same as in Theorem~\ref{thm:Sion-CAT0}. 
 If 
 $x_1,x_2, \dots , x_m\in Y$ and $\alpha \in \R$ satisfy 
 \begin{align}
  \alpha < \min_{y\in X}\max_{1\leq k\leq m}f(x_k, y), 
 \end{align}
 then there exists $x_0\in X$ such that 
 $\alpha < \min_{y\in Y}f(x_0,y)$. 
\end{lemma}

\begin{proof}
 The proof is by induction on $m\in \N$. 
 If $m=1$, then the result clearly holds. 
 Suppose that the result holds for some $m\in \N$. 
 Let $x_1,x_2, \dots , x_m, x_{m+1}\in X$ and $\alpha \in \R$ satisfy 
 \begin{align}
  \alpha < \min_{y\in Y}\max_{1\leq k\leq m+1}f(x_k, y). 
 \end{align}
 Set 
 \begin{align}
  Y'=\{y\in Y: f(x_{m+1}, y)\leq \alpha\}. 
 \end{align}
 Then $Y'$ is a closed and convex subset of $Y$. 
 If $Y'$ is empty, then 
 \begin{align}
  f(x_{m+1},y) > \alpha
 \end{align}
 for all $y\in Y$ and hence we have 
 $\min_{y\in Y}f(x_{m+1},y)>\alpha$. 
 Letting $x_0=x_{m+1}$, we obtain the conclusion. 
 Thus we may suppose that $Y'$ is nonempty. Then 
 $Y'$ is a bounded Hadamard space. Since 
 $f(x_{m+1},y)\leq \alpha$ for all $y\in Y'$, we have 
 \begin{align}
  \alpha 
 < \min_{y\in Y} \max_{1\leq k\leq m+1} f(x_k,y) 
 \leq \min_{y\in Y'} \max_{1\leq k\leq m+1} f(x_k,y) 
 = \min_{y\in Y'} \max_{1\leq k\leq m} f(x_k,y). 
 \end{align}
 It follows from the assumption of our induction argument that 
 there exists $x'_0\in X$ such that 
 $\alpha < \min_{y\in Y'}f(x'_0,y)$.  
 If $y\in Y'$, then $\alpha < f(x'_0,y)$ and $f(x_{m+1},y)\leq \alpha$. 
 Hence we have 
 \begin{align}
  \max\{f(x'_0,y),f(x_{m+1},y)\} 
 =f(x'_0,y) > \alpha. 
 \end{align}
 If $y\in Y\setminus Y'$, then 
 \begin{align}
  \max\{f(x'_0,y), f(x_{m+1},y)\} \geq f(x_{m+1},y) >\alpha. 
 \end{align}
 Thus 
 \begin{align}
  \max\{f(x'_0,y), f(x_{m+1},y)\} > \alpha
 \end{align}
 for all $y\in Y$. Consequently, we have 
 \begin{align}
  \min_{y\in Y} \max\{f(x'_0,y), f(x_{m+1},y)\} > \alpha. 
 \end{align}
 Lemma~\ref{lem:Sion-CAT0-1} ensures that 
 there exists $x_0\in X$ such that 
 $\min_{y\in Y}f(x_0,y)>\alpha$.  
 This completes the proof. 
\end{proof}

\begin{proof}[The proof of Theorem~\ref{thm:Sion-CAT0}]
 As is noted in Remark~\ref{rem:Sion-CAT0}, we have 
 \begin{align}
  \sup_{x\in X}\min_{y\in Y}f(x,y)\leq \min_{y\in Y}\sup_{x\in X}f(x,y). 
 \end{align}
 Thus we prove the reverse inequality. 
 Let $\alpha \in \R$ satisfy 
 $\alpha <\min_{y\in Y}\sup_{x\in X}f(x,y)$ 
 and set 
 \begin{align}
  Y_x=\bigl\{y\in Y: f(x,y)\leq \alpha\bigr\}
 \end{align}
 for all $x\in X$. Then we can prove that 
 $\bigcap _{x\in X}Y_x$ is empty. 
 If not, then there exists $y_0\in Y$ such that 
 $y_0\in Y_x$ for all $x\in X$ and hence 
 \begin{align}
  \alpha \geq \sup_{x\in X}f(x, y_0)
 \geq \min_{y\in Y}\sup_{x\in X}f(x,y). 
 \end{align}
 This is a contradiction. Hence $\bigcap _{x\in X}Y_x$ is empty. 
 Applying Lemma~\ref{lem:Delta-compact}, we have 
 $m\in \N$ and $x_1,x_2,\dots ,x_m\in X$ such that 
 $\bigcap_{k=1}^{m}Y_{x_k}=\emptyset$.  
 Since 
 \begin{align}
  \bigcap_{k=1}^{m}Y_{x_k} 
  = \left\{
 y\in Y: \max_{1\leq k\leq m}f(x_k,y)\leq \alpha
 \right\}, 
 \end{align}
 we have $\alpha < \max_{1\leq k\leq m}f(x_k,y)$ for all $y\in Y$ 
 and hence 
 \begin{align}
  \alpha < \min_{y\in Y}\max_{1\leq k\leq m}f(x_k,y). 
 \end{align}
 Then Lemma~\ref{lem:Sion-CAT0-2} ensures that 
 there exists $x_0\in X$ such that 
 $\alpha < \min_{y\in Y} f(x_0,y)$.  
 Consequently, we have 
 \begin{align}
  \alpha < \sup_{x\in X}\min_{y\in Y} f(x,y).  
 \end{align}
 Letting $\alpha \uparrow \min_{y\in Y}\sup_{x\in X}f(x,y)$ yields 
 \begin{align}
  \min_{y\in Y}\sup_{x\in X}f(x,y)
  \leq \sup_{x\in X}\min_{y\in Y} f(x,y). 
 \end{align} 
 This completes the proof.   
\end{proof}

As a direct consequence of Theorem~\ref{thm:Sion-CAT0}, 
we obtain the following saddle point theorem: 

\begin{corollary}\label{cor:Sion-CAT0}
 Let $(X,d_X)$ and $(Y,d_Y)$ be bounded Hadamard spaces and 
 $f\colon X\times Y\to \R$ a function 
 satisfying~(i) and~(ii) in 
 Theorem~\ref{thm:Sion-CAT0}. 
 Then the minimax equality 
 \begin{align}
  \max_{x\in X}\min_{y\in Y} f(x,y)=\min_{y\in Y} \max_{x\in X} f(x,y) 
 \end{align}
 holds and $f$ has a saddle point. 
\end{corollary}

\begin{proof}
 Since $X$ and $Y$ are bounded Hadamard space, 
 Lemma~\ref{lem:Delta-bdd} 
 and Theorem~\ref{thm:Sion-CAT0} ensure that the minimax 
 equality holds. Then there exists $(x_0,y_0) \in X\times Y$ such that 
 \begin{align}
 \min_{y\in Y}f(x_0,y) = 
  \max_{x\in X}\min_{y\in Y} f(x,y) 
 \quad 
 \textrm{and}
 \quad 
 \max_{x\in X} f(x,y_0)= \min_{y\in Y} \max_{x\in X} f(x,y). 
 \end{align}
 These equalities give us that 
 \begin{align}
  f(x,y_0) \leq \max_{x\in X}f(x, y_0) 
 =\min_{y\in Y}f(x_0,y) \leq f(x_0,y)
 \end{align}
 for all $(x,y)\in X\times Y$. 
 Consequently, we obtain the desired inequalities. 
\end{proof}

\section{Coercive saddle functions in Hadamard spaces}\label{sec:coercive}

In this section, we prove the following minimax theorem 
for coercive saddle functions in Hadamard spaces 
without boundedness assumption: 

\begin{theorem}\label{thm:minimax-coercive}
 Let $(X,d_X)$ and $(Y,d_Y)$ be Hadamard spaces, 
 $d$ the $\ell^2$-metric on $X\times Y$, 
 and $f\colon X\times Y\to \R$ a function satisfying 
 \begin{enumerate}
  \item[(i)] $f(\,\cdot\,,y)$ is upper semicontinuous 
 and concave for each $y\in Y$; 
  \item[(ii)] $f(x,\cdot \,)$ is lower semicontinuous 
 and convex for each $x\in X$; 
  \item[(iii)] there exists $(a,b)\in X\times Y$ such that 
 \begin{align}
  d\bigl((x_n,y_n), (a,b)\bigr)\to \infty 
  \Rightarrow \liminf_{n\to \infty}\bigl(
  f(x_n, b) - f(a, y_n)\bigr) <0
 \end{align}
 whenever $\{(x_n,y_n)\}$ is a sequence in $X\times Y$. 
 \end{enumerate}
 Then the minimax equality 
 \begin{align}
  \max_{x\in X}\min_{y\in Y} f(x,y)=\min_{y\in Y} \max_{x\in X} f(x,y) 
 \end{align}
 holds and $f$ has a saddle point. 
\end{theorem}

\begin{proof}
 Note that $(X\times Y, d)$ is an Hadamard space 
 and the unique normalized geodesic $\gamma$ from $(x,y)$ to $(x',y')$ 
 in $X\times Y$ 
 is given by~\eqref{eq:geodesic-product}. 
 We also define the $\ell^{\infty}$-metric $d_{\infty}$ on $X\times Y$ by 
 \begin{align}
  d_{\infty}\bigl((x,y), (x',y')\bigr)
 =\max\left\{
  d_X(x,x'), d_Y(y,y')
  \right\}
 \end{align}
 for all $(x,y)$ and $(x',y')$ in $X\times Y$. 
 It is clear that 
 \begin{align}
  d_{\infty}\bigl((x,y), (x',y')\bigr)
  \leq d\bigl((x,y), (x',y')\bigr)
  \leq \sqrt{2} \, d_{\infty}\bigl((x,y), (x',y')\bigr)
 \end{align}
 and hence the metric topology on $(X\times Y,d)$ coincides with 
 that on $(X\times Y,d_{\infty})$. 
 Let $F$ be the function defined by 
 \begin{align}
  F\bigl((x,y), (x',y')\bigr) = f(x, y') - f(x', y)
 \end{align}
 for all $(x,y)$ and $(x',y')$ in $X\times Y$. 

 It follows from~(iii) that there exists $\delta>0$ such that 
 \begin{align}
  d_{\infty}\bigl((x,y), (a,b)\bigr) \geq \delta 
 \Rightarrow 
 f(x, b) - f(a, y) \leq 0. 
 \end{align}
 If not, then 
 we have a sequence $\{(x_n,y_n)\}$ in $X\times Y$ such that 
 \begin{align}
  d_{\infty}\bigl((x_n,y_n),(a,b)\bigr)\to \infty 
  \quad \textrm{and} \quad 
  f(x_n, b) - f(a, y_n) > 0. 
 \end{align}
 This gives us that 
 \begin{align}
  \liminf_{n\to \infty}\bigl(f(x_n, b) - f(a, y_n)\bigr)\geq 0. 
 \end{align}
 On the other hand, since 
 \begin{align}
  d\bigl((x_n,y_n),(a,b)\bigr)
  \geq d_{\infty}\bigl((x_n,y_n),(a,b)\bigr)
  \to \infty, 
 \end{align}
 we have from~(iii) that 
 \begin{align}
  \liminf_{n\to \infty}\bigl(f(x_n, b) - f(a, y_n)\bigr) <0. 
 \end{align}
 This is a contradiction. 

 In order to apply Corollary~\ref{cor:Sion-CAT0}, 
 we define two sets $C$ and $D$ by 
 \begin{align}
  C=\{x\in X:d_X(x,a)\leq \delta\}
 \quad \textrm{and} \quad 
  D=\{y\in Y: d_Y(y,b)\leq \delta\}. 
 \end{align}
 It is obvious that 
 \begin{align}
  C\times D = \{(x,y)\in X\times Y: d_{\infty}
 \bigl((x,y), (a,b)\bigr) \leq \delta\}. 
 \end{align}
 Since $C$ and $D$ are bounded closed convex subsets of 
 $X$ and $Y$, respectively, these spaces are bounded Hadamard spaces. 
 Thus Corollary~\ref{cor:Sion-CAT0} ensures that 
 there a saddle point $(x_0, y_0) \in C\times D$ 
 of the function $f$ on $C\times D$. 
 This is equivalent to 
 \begin{align}\label{eq:thm:minimax-coercive-a}
  F\bigl((x_0,y_0), (x,y)\bigr) 
 = f(x_0, y) - f(x, y_0) \geq 0
 \end{align}
 for all $(x,y) \in C\times D$. 

 We next prove that the inequality~\eqref{eq:thm:minimax-coercive-a} holds 
 for all $(x,y) \in X\times Y$. 
 Let $(x,y)$ be a point in $X\times Y$. 
 Since $(x_0,y_0)\in C\times D$, we have 
 $d_{\infty} \bigl((x_0,y_0), (a,b)\bigr) \leq \delta$.  
 We first consider the case where $d_{\infty}\bigl((x_0,y_0), (a,b)\bigr)<\delta$. 
 Since the function 
 \begin{align}
  t\mapsto (1-t) (x_0,y_0) \oplus t (x,y)
 \end{align}
 on $[0,1]$ is convergent to $(x_0,y_0)$ as $t\downarrow 0$ and 
 the set 
 \begin{align}
  V=\bigl\{(u,v)\in X\times Y: 
  d_{\infty} \bigl((u,v), (a,b)\bigr) < \delta
 \bigr\}
 \end{align}
 is an open neighborhood of $(x_0,y_0)$, 
 there exists $t_0\in (0,1)$ such that 
 \begin{align}
  (1-t_0) (x_0,y_0) \oplus t_0 (x,y) 
  \in V. 
 \end{align}
 Since $V\subset C\times D$ and $F\bigl((x_0,y_0),\cdot\,\bigr)$ is convex 
 on $X\times Y$, we have 
 \begin{align}
 \begin{split}
   0 
  &\leq F\bigl((x_0,y_0), (1-t_0) (x_0,y_0) \oplus t_0 (x,y) \bigr) \\
  &\leq (1-t_0) F\bigl((x_0,y_0), (x_0,y_0)\bigr) 
 +t_0 F\bigl((x_0,y_0), (x,y)\bigr) \\
  &\leq F\bigl((x_0,y_0), (x,y) \bigr).   
 \end{split}
 \end{align}
 Thus~\eqref{eq:thm:minimax-coercive-a} holds. 
 We next consider the case where 
 $d_{\infty}\bigl((x_0,y_0), (a,b)\bigr)=\delta$.  
 Then we have 
 \begin{align}
  F\bigl((x_0,y_0), (a,b)\bigr)
  =f(x_0,b) - f(a, y_0) \leq 0. 
 \end{align}
 Since $(a,b)\in C\times D$ and $(x_0,y_0)$ 
 is a saddle point of $f$ on $C\times D$, 
 we also have 
  $F\bigl((x_0,y_0), (a,b)\bigr) \geq 0$. 
 Consequently, we obtain 
 \begin{align}
  F\bigl((x_0,y_0), (a,b)\bigr)=0. 
 \end{align}
 Since the set $V$ defined above is an open neighborhood of $(a,b)$, 
 we can choose $t_1\in (0,1)$ such that 
 \begin{align}
  (1-t_1) (a,b) \oplus t_1 (x,y)  
  \in V. 
 \end{align}
 Then we have 
 \begin{align}
 \begin{split}
  0 
  &\leq F\bigl((x_0,y_0), (1-t_1) (a,b) \oplus t_1 (x,y) \bigr) \\
  &\leq (1-t_1)F\bigl((x_0,y_0), (a,b) \bigr) 
 + t_1 F\bigl((x_0,y_0), (x,y) \bigr) \\
  &\leq  F\bigl((x_0,y_0), (x,y) \bigr).   
 \end{split}
 \end{align}
 Thus~\eqref{eq:thm:minimax-coercive-a} holds. 
 Consequently, $(x_0,y_0)$ is a saddle point of $f$ on $X\times Y$. 

 We finally prove the minimax equality. 
 Since $(x_0,y_0)$ is a saddle point of $f$ on $X\times Y$, we have 
 \begin{align}
  \max_{x\in X}f(x,y_0) = f(x_0,y_0) = \min_{y\in Y} f(x_0,y). 
 \end{align}
 This gives us that 
 \begin{align}
  \min_{y\in Y}\max_{x\in X}f(x,y)
  \leq \max_{x\in X}f(x,y_0) = 
  \min_{y\in Y} f(x_0,y)
  \leq \max_{x\in X}\min_{y\in Y} f(x,y). 
 \end{align}
 and hence we obtain 
 \begin{align}
  \min_{y\in Y}\max_{x\in X}f(x,y)
  \leq \max_{x\in X}\min_{y\in Y} f(x,y). 
 \end{align}
 The reverse inequality is obvious. This completes the proof. 
\end{proof}

\section{Resolvents of saddle functions in Hadamard spaces}\label{sec:resolvent}

In this section, we define the resolvents of 
saddle functions in Hadamard spaces. 

\begin{theorem}\label{thm:resolvent-minimax}
 Let $(X,d_X)$ and $(Y,d_Y)$ be Hadamard spaces 
 and $f\colon X\times Y\to \R$ a function satisfying 
 \begin{enumerate}
  \item[(i)] $f(\,\cdot\,,y)$ is upper semicontinuous 
 and concave for each $y\in Y$; 
  \item[(ii)] $f(x,\cdot \,)$ is lower semicontinuous 
 and convex for each $x\in X$. 
 \end{enumerate}
 If $(a,b)\in X\times Y$, then 
 there exists a unique saddle point $(\hat{a}, \hat{b})$ 
 in $X\times Y$ of the function $g\colon X\times Y\to \R$ 
 defined by 
\begin{align}
  g(x,y)= f(x,y) - \frac{1}{2}d_X(x,a)^2 + \frac{1}{2}d_Y(y,b)^2 
\end{align}
 for all $(x,y)\in X\times Y$. 
\end{theorem}

\begin{proof}
 Let $d$ be the $\ell^2$-metric on $X\times Y$. 
 It is clear that $g(\,\cdot\,, y)$ is 
 upper semicontinuous and concave for each $y\in Y$ 
 and $g(x,\cdot \,)$ is lower semicontinuous and convex 
 for each $x\in X$. 
 
 We then prove that 
 \begin{align}
  d\bigl((x_n,y_n), (a,b)\bigr)\to \infty 
  \Rightarrow \frac{g(x_n, b) - g(a, y_n)}
  {d\bigl((x_n,y_n), (a,b)\bigr)} 
  \to -\infty
 \end{align}
 whenever $\{(x_n,y_n)\}$ is a sequence in $X\times Y$. 
 Let $\{(x_n,y_n)\}$ be a sequence in $X\times Y$ such that 
 \begin{align}
  d\bigl((x_n,y_n), (a,b)\bigr)\to \infty. 
 \end{align}
 Note that the function $\Phi\colon X\times Y\to \R$ defined by 
 \begin{align}
  \Phi(x,y) = f(a,y) - f(x,b)
 \end{align} 
 for all $(x,y)\in X\times Y$ 
 is a real-valued lower semicontinuous convex function 
 on the Hadamard space $(X\times Y,d)$. 
 Theorem~\ref{thm:Bacak-umbrella} ensures that 
 there exists a positive real number $C$ such that 
 \begin{align}
  \Phi(x,y) \geq -C\Bigl[d\bigl((x,y),(a,b)\bigr)+1\Bigr]
 \end{align}
 for all $(x,y)\in X\times Y$. This gives us that 
 \begin{align}
 \begin{split}
  \frac{g(x_n, b) - g(a, y_n)}{d\bigl((x_n,y_n), (a,b)\bigr)} 
  &= \frac{f(x_n, b) - f(a, y_n)
  - 2^{-1}d_X(x_n,a)^2 - 2^{-1}d_Y(y_n,b)^2
  }{d\bigl((x_n,y_n), (a,b)\bigr)} \\
  &= \frac{-\Phi(x_n,y_n) 
  - 2^{-1} d\bigl((x_n,y_n), (a,b)\bigr)^2}{d\bigl((x_n,y_n), (a,b)\bigr)} \\
  &\leq C + \frac{1}{d\bigl((x_n,y_n), (a,b)\bigr)} 
  -\frac{1}{2}d\bigl((x_n,y_n), (a,b)\bigr) \to -\infty.   
 \end{split}
 \end{align}
 This gives us that 
\begin{align}
 \frac{g(x_n, b) - g(a, y_n)}
  {d\bigl((x_n,y_n), (a,b)\bigr)} 
  \to -\infty. 
\end{align}
 Consequently, we have 
 \begin{align}
  d\bigl((x_n,y_n), (a,b)\bigr)\to \infty 
  \Rightarrow \liminf_{n\to \infty}\bigl(g(x_n, b) - g(a, y_n)\bigr) <0
 \end{align}
 whenever $\{(x_n,y_n)\}$ is a sequence in $X\times Y$. 
 Therefore Theorem~\ref{thm:minimax-coercive} ensures that 
 $g$ has a saddle point $(\hat{a},\hat{b})$. 

 We finally prove the uniqueness of such a saddle point.  
 Let $(a_i,b_i)$ be saddle points of $g$ for $i=1,2$. 
 Then we have 
 \begin{align}
  g(a_i,y) - g(x,b_i) \geq 0 
 \end{align}
 for all $(x,y)\in X\times Y$ and $i=1,2$. 
 Set
 \begin{align}
  (p,q)=\frac{1}{2}(a_1,b_1)\oplus \frac{1}{2}(a_2,b_2)
 \end{align}
 and define $F$ and $G$ by 
 \begin{align}
 \begin{split}
  F\bigl((x,y),(x',y')\bigr) &= f(x,y')-f(x',y); \\
  G\bigl((x,y),(x',y')\bigr) &= g(x,y')-g(x',y)  
 \end{split}
 \end{align}
 for all $(x,y)$ and $(x',y')$ in $X\times Y$. 
 It is clear that 
 \begin{align}
 \begin{split}
  &G\bigl((x,y),(x',y')\bigr) \\
  &\quad = F\bigl((x,y),(x',y')\bigr) 
  -\frac{1}{2}d\bigl((x,y), (a,b)\bigr)^2 
  +\frac{1}{2}d\bigl((x',y'), (a,b)\bigr)^2  
 \end{split}
 \end{align}
for all $(x,y)$ and $(x',y')$ in $X\times Y$. 
 We then obtain 
 \begin{align}
 \begin{split}
  0 
  &\leq 2G\bigl((a_i,b_i),(p,q)\bigr) \\
  &= 2F\bigl((a_i,b_i),(p,q)\bigr)
 -d\bigl((a_i,b_i), (a,b)\bigr)^2 
  +d\bigl((p,q), (a,b)\bigr)^2 \\
  &\leq F\bigl((a_i,b_i),(a_1,b_1)\bigr) 
 + F\bigl((a_i,b_i),(a_2,b_2)\bigr)  -d\bigl((a_i,b_i), (a,b)\bigr)^2 \\
  &\quad +\frac{1}{2}d\bigl((a_1,b_1), (a,b)\bigr)^2 
+\frac{1}{2}d\bigl((a_2,b_2), (a,b)\bigr)^2 
 -\frac{1}{4}d\bigl((a_1,b_1), (a_2,b_2)\bigr)^2.   
 \end{split}
 \end{align}
 Summing up these two inequalities yield that 
 \begin{align}
 \begin{split}
  0&\leq 
 F\bigl((a_1,b_1),(a_2,b_2)\bigr) 
 + F\bigl((a_2,b_2),(a_1,b_1)\bigr) \\
 &\quad -d\bigl((a_1,b_1), (a,b)\bigr)^2 
-d\bigl((a_2,b_2), (a,b)\bigr)^2 \\
  &\quad +d\bigl((a_1,b_1), (a,b)\bigr)^2 
+d\bigl((a_2,b_2), (a,b)\bigr)^2 
 -\frac{1}{2}d\bigl((a_1,b_1), (a_2,b_2)\bigr)^2 \\
  &= -\frac{1}{2}d\bigl((a_1,b_1), (a_2,b_2)\bigr)^2  
 \end{split}
 \end{align}
Hence we have $(a_1,b_1)=(a_2,b_2)$. 
This completes the proof. 
\end{proof}

We next give the definition of the resolvents of saddle functions  
in Hadamard spaces. 

\begin{definition}
 \label{def:resolvent-of-saddle-function}
 Let $(X,d_X)$, $(Y,d_Y)$, 
 and $f\colon X\times Y\to \R$ be the same as in 
 Theorem~\ref{thm:resolvent-minimax}. 
 For each point $(x,y)$ in $X\times Y$, 
 the unique saddle point 
 $(\hat{x},\hat{y})$ of the function 
 \begin{align}
  (u,v) \mapsto  f(u,v) -\frac{1}{2}d_X(u,x)^2 + \frac{1}{2}d_Y(v,y)^2
 \end{align}
 on $X\times Y$ is denoted by $R_f(x,y)$. The mapping 
 $R_f$ of $X\times Y$ into itself is called the resolvent of $f$. 
\end{definition}

\begin{theorem}\label{thm:resolvent-saddle}
 Let $(X,d_X)$, $(Y,d_Y)$, and $f$ be the same as in 
 Theorem~\ref{thm:resolvent-minimax}, 
 $d$ the $\ell^2$-metric on $X\times Y$,  
 $R_{\lambda}$ the resolvent of $\lambda f$ for each $\lambda >0$, 
 and set 
 \begin{align}
  R_{\lambda}(x,y)=\bigl(R_{1,\lambda}(x,y),R_{2,\lambda}(x,y)\bigr)
 \end{align}
 for each $(x,y)\in X\times Y$ and $\lambda >0$. 
 Then the following hold: 
 \begin{enumerate}
  \item[(i)] If $\lambda>0$, then $R_{\lambda}$ is a single-valued mapping of 
 $X\times X$ into itself and $\Fix (R_{\lambda})=\Saddle (f)$; 
  \item[(ii)] if $\lambda>0$, then 
 the inequality 
 \begin{align}
  \begin{split}
  &d\bigl(R_{\lambda}(x,y), (x',y')\bigr)^2
 +d\bigl(R_{\lambda}(x,y), (x,y)\bigr)^2 \\
 &\quad +2\lambda \bigl\{f(x',R_{2,\lambda}(x,y))-f(R_{1,\lambda}(x,y),y')\bigr\} 
  \leq d\bigl((x,y), (x',y')\bigr)^2   
  \end{split}
 \end{align}
  holds for all $(x,y)$ and $(x',y')$ in $X\times Y$; 
  \item[(iii)] if $\lambda, \mu>0$, then 
 the inequality 
 \begin{align}
 \begin{split}
  &(\lambda+\mu)d\bigl(R_{\lambda}(x,y), R_{\mu}(x',y')\bigr)^2
 +\mu d\bigl(R_{\lambda}(x,y), (x,y)\bigr)^2 \\
 &\quad + \lambda d\bigl(R_{\mu}(x',y'), (x',y')\bigr)^2 \\
 &\leq \lambda d\bigl(R_{\lambda}(x,y), (x',y')\bigr)^2
+\mu d\bigl(R_{\mu}(x',y'), (x,y)\bigr)^2  
 \end{split}
 \end{align}
  holds for all $(x,y)$ and $(x',y')$ in $X\times Y$; 
  \item[(iv)] if $\lambda >0$, then $R_{\lambda}$ 
 is both firmly metrically nonspreading and 
 nonexpansive; 
  \item[(v)] if $\lambda,\mu>0$ and $(x,y)\in X\times Y$, then 
 \begin{align}
  \frac{1}{\mu}d\bigl(R_{\mu}R_{\lambda}(x,y),R_{\lambda}(x,y)\bigr)
  \leq \frac{1}{\lambda}d\bigl(R_{\lambda}(x,y),(x,y)\bigr). 
 \end{align}
 \end{enumerate}
\end{theorem}

\begin{proof}
 For each $\lambda>0$ and $(x,y)\in X\times Y$, 
 we define $g_{\lambda, x,y}\colon X\times Y\to \R$ by 
 \begin{align}
  g_{\lambda, x,y}(u,v)
  = f(x,y) -\frac{1}{2\lambda}d_X(u,x)^2 + \frac{1}{2\lambda}d_Y(v,y)^2
 \end{align}
 for all $(u,v) \in X\times Y$. 
 We define a real functions $F$ and $G_{\lambda,x,y}$ 
 by 
 \begin{align}
  \begin{split}
  F\bigl((x,y),(x',y')\bigr) &= f(x,y')-f(x',y); \\
  G_{\lambda,x,y}\bigl((u,v),(u',v')\bigr) & = g_{\lambda,x,y}(u,v') - 
 g_{\lambda,x,y}(u',v)   
  \end{split}
 \end{align}
 for all $(u,v)$ and $(u',v')$ in $X\times Y$. 
 We denote by $d$ the $\ell^2$-metric on $X\times Y$. Then we have 
 \begin{align}
 \begin{split}
  &G_{\lambda,x,y}\bigl((u,v),(u',v')\bigr) \\
  &\quad = F\bigl((u,v),(u',v')\bigr) 
  -\frac{1}{2\lambda}d\bigl((u,v),(x,y)\bigr)^2 
  + \frac{1}{2\lambda}d\bigl((u',v'), (x,y)\bigr)^2  
 \end{split}
 \end{align}
 for all $(u,v)$ and $(u',v')$ in $X\times Y$. 

 We first prove the part~(i). 
 Let $\lambda>0$. It is clear from the definition of $R_{\lambda}$ that 
 $R_{\lambda}$ is a single-valued mapping of $X\times Y$ 
 into itself. Let $(\bar{x},\bar{y})\in \Fix(R_{\lambda})$ be given. 
 Then we have 
 \begin{align}
  0&\leq G_{\lambda, \bar{x},\bar{y}}\bigl((\bar{x},\bar{y}),(u,v)\bigr)
 \end{align}
 for all $(u,v)\in X\times Y$. Let $(u,v)\in X\times Y$ and 
 set 
 \begin{align}
  (z_t,w_t) = (1-t)(\bar{x},\bar{y}) \oplus (1-t)(u,v)
 \end{align}
 for all $t\in [0,1]$. Then we have 
 \begin{align}
\begin{split}
  0
&\leq G_{\lambda, \bar{x},\bar{y}}\bigl((\bar{x},\bar{y}),(z_t,w_t)\bigr) \\
&\leq (1-t)F\bigl((\bar{x},\bar{y}),(\bar{x},\bar{y})\bigr) 
+tF\bigl((\bar{x},\bar{y}),(u,v)\bigr) 
 +\frac{t^2}{2\lambda}d\bigl((\bar{x},\bar{y}), (u,v)\bigr)^2 \\
&\leq tF\bigl((\bar{x},\bar{y}),(u,v)\bigr) 
 +\frac{t^2}{2\lambda}d\bigl((\bar{x},\bar{y}), (u,v)\bigr)^2 
\end{split}
 \end{align} 
 and hence 
 \begin{align}
  0\leq F\bigl((\bar{x},\bar{y}),(u,v)\bigr) 
 +\frac{t}{2\lambda}d\bigl((\bar{x},\bar{y}), (u,v)\bigr)^2. 
 \end{align}
 Letting $t\downarrow 0$, we have 
 $0\leq F\bigl((\bar{x},\bar{y}),(u,v)\bigr)$. Thus $(\bar{x},\bar{y})\in \Saddle(f)$. 
 Conversely, if $(\bar{x},\bar{y})\in \Saddle(f)$, then we have 
 \begin{align}
  0
 &\leq F\bigl((\bar{x},\bar{y}),(u,v)\bigr) 
 \leq G_{\lambda, \bar{x},\bar{y}}\bigl((\bar{x},\bar{y}),(u,v)\bigr)
 \end{align}
 for all $(u,v)\in X\times Y$. Hence $R_{\lambda}(\bar{x},\bar{y})=(\bar{x},\bar{y})$. 
 Consequently, we obtain $(\bar{x},\bar{y})\in \Fix(R_{\lambda})$. 

 We next prove the part~(ii). Let $\lambda>0$ be given 
 and let $(x,y)$ and $(x',y')$ be points in $X\times Y$. 
 Set 
 \begin{align}
  (z_t,w_t) = (1-t)R_{\lambda}(x,y) \oplus t(x',y')
 \end{align}
 for all $t\in [0,1]$. Then we have 
 \begin{align}
 \begin{split}
  0
 &\leq G_{\lambda,x,y}\bigl(R_{\lambda}(x,y),(z_t,w_t)\bigr) \\
 &\leq (1-t)F\bigl(R_{\lambda}(x,y),R_{\lambda}(x,y)\bigr) 
 + tF\bigl(R_{\lambda}(x,y),(x',y')\bigr) \\
 &\quad -\frac{1}{2\lambda}d\bigl(R_{\lambda}(x,y),(x,y)\bigr)^2 
 +\frac{1}{2\lambda}d\bigl((z_t,w_t),(x,y)\bigr)^2 \\
 &\leq tF\bigl(R_{\lambda}(x,y),(x',y')\bigr) 
 -\frac{t}{2\lambda}d\bigl(R_{\lambda}(x,y),(x,y)\bigr)^2 
 +\frac{t}{2\lambda}d\bigl((x',y'),(x,y)\bigr)^2 \\
 &\quad 
-\frac{t(1-t)}{2\lambda}d\bigl(R_{\lambda}(x,y),(x',y')\bigr)^2  
 \end{split}
 \end{align}
 and hence 
 \begin{align}
 \begin{split}
 &(1-t)d\bigl(R_{\lambda}(x,y),(x',y')\bigr)^2 
 + d\bigl(R_{\lambda}(x,y),(x,y)\bigr)^2 
 + 2\lambda F\bigl((x',y'),R_{\lambda}(x,y)\bigr) \\
 &\quad \leq d\bigl((x',y'),(x,y)\bigr)^2.   
 \end{split}
 \end{align}
 Letting $t\downarrow 0$, we obtain the conclusion. 

 We next prove the part~(iii). Let $\lambda, \mu>0$ be given. 
 It follows from the part~(ii) that 
 \begin{align}
 \begin{split}
 &\mu d\bigl(R_{\lambda}(x,y),R_{\mu}(x',y')\bigr)^2 
 + \mu d\bigl(R_{\lambda}(x,y),(x,y)\bigr)^2 \\
 &\quad + 2\lambda \mu F\bigl(R_{\mu}(x',y'),R_{\lambda}(x,y)\bigr) 
 \leq \mu d\bigl(R_{\mu}(x',y'),(x,y)\bigr)^2  
 \end{split}
 \end{align}
 and 
 \begin{align}
 \begin{split}
 &\lambda d\bigl(R_{\mu}(x',y'),R_{\lambda}(x,y)\bigr)^2 
 + \lambda d\bigl(R_{\mu}(x',y'),(x',y')\bigr)^2 \\
 &\quad + 2\lambda \mu F\bigl(R_{\lambda}(x,y),R_{\mu}(x',y')\bigr) 
 \leq \lambda d\bigl(R_{\lambda}(x,y),(x',y')\bigr)^2.   
 \end{split}
 \end{align}
 Summing up these inequalities yields the conclusion. 

 We next prove the part~(iv). Let $\lambda>0$ be given. 
 Then we have from the part~(iii) that 
 \begin{align}
 \begin{split}
  &2d\bigl(R_{\lambda}(x,y), R_{\lambda}(x',y')\bigr)^2
 +d\bigl(R_{\lambda}(x,y), (x,y)\bigr)^2 
 +d\bigl(R_{\lambda}(x',y'), (x',y')\bigr)^2 \\
  &\quad \leq d\bigl(R_{\lambda}(x,y), (x',y')\bigr)^2
+ d\bigl(R_{\lambda}(x',y'), (x,y)\bigr)^2  
 \end{split}
 \end{align}
  holds for all $(x,y)$ and $(x',y')$ in $X\times Y$. 
 This means that $R_{\lambda}$ is firmly metrically nonspareading. 
 If $(x,y)$ and $(x',y')$ are points in $X\times Y$, 
 then Theorem~\ref{thm:CS-inequality} implies that 
  \begin{align}
 \begin{split}
  &d\bigl(R_{\lambda}(x,y), R_{\lambda}(x',y')\bigr)^2 \\
  &\quad \leq \frac{1}{2}\bigl\{d\bigl(R_{\lambda}(x,y), (x',y')\bigr)^2
+ d\bigl(R_{\lambda}(x',y'), (x,y)\bigr)^2
 -d\bigl(R_{\lambda}(x,y), (x,y)\bigr)^2 \\
 &\qquad -d\bigl(R_{\lambda}(x',y'), (x',y')\bigr)^2\bigr\} \\
 &\quad \leq d\bigl(R_{\lambda}(x,y), R_{\lambda}(x',y')\bigr)
d\bigl((x,y), (x',y')\bigr)  
 \end{split}
 \end{align}
 This implies that $R_{\lambda}$ is nonexpansive. 

 We finally prove the part~(v). 
 Let $\lambda, \mu>0$ and $(x,y)\in X\times Y$ be given. 
 It follows from the part~(iii) that 
 \begin{align}
 \begin{split}
  &(\lambda+\mu)d\bigl(R_{\lambda}(x,y), R_{\mu}R_{\lambda}(x,y)\bigr)^2
 +\mu d\bigl(R_{\lambda}(x,y), (x,y)\bigr)^2 \\
  &\quad +\lambda d\bigl(R_{\mu}R_{\lambda}(x,y), R_{\lambda}(x,y)\bigr)^2 \\
  &\leq \lambda d\bigl(R_{\lambda}(x,y), R_{\lambda}(x,y)\bigr)^2
+\mu d\bigl(R_{\mu}R_{\lambda}(x,y), (x,y)\bigr)^2.   
 \end{split}
 \end{align} 
 Using Theorem~\ref{thm:CS-inequality}, we have 
 \begin{align}
 \begin{split}
  &\lambda d\bigl(R_{\mu}R_{\lambda}(x,y), R_{\lambda}(x,y)\bigr)^2 \\
  &\leq 
  \frac{\mu}{2}\bigl\{
  d\bigl(R_{\mu}R_{\lambda}(x,y), (x,y)\bigr)^2 
 - d\bigl(R_{\mu}R_{\lambda}(x,y),R_{\lambda}(x,y)\bigr)^2 \\
 &\quad - d\bigl(R_{\lambda}(x,y), (x,y)\bigr)^2 
  \bigr\} \\
  &\leq \mu d\bigl(R_{\mu}R_{\lambda}(x,y), R_{\lambda}(x,y)\bigr) 
d\bigl(R_{\lambda}(x,y), (x,y)\bigr)  
 \end{split}
 \end{align} 
 and hence we obtain the conclusion. 
\end{proof}

Using Theorems~\ref{thm:metrically-nonspareding} 
and~\ref{thm:resolvent-saddle}, we obtain the following: 

\begin{corollary}
 Let $(X,d_X)$, $(Y,d_Y)$, $f$ be the same as in 
 Theorem~\ref{thm:resolvent-minimax} and 
 $R_f$ the resolvent of $f$. 
 Then $\Saddle (f)$ is nonempty if and only if 
 $\{R_f^n(x,y)\}$ is bounded for some $(x,y)\in X\times Y$. 
 In this case, $\{R_f^n(x,y)\}$ is $\Delta$-convergent to 
 an element of $\Saddle (f)$ 
 for each $(x,y)\in X\times Y$.  
\end{corollary}

\section{The proximal point algorithm for saddle functions}\label{sec:ppa}

In this section, we study the asymptotic behavior of 
sequences generated by the proximal point algorithm 
for saddle functions in Hadamard spaces. 

Applying the results obtained in this papers, we obtain 
the following result generalizing Theorem~\ref{thm:Rockafellar-ppa-saddle} 
to the Hadamard space setting: 

\begin{theorem}\label{thm:PPA-Saddle}
 Let $(X,d_X)$, $(Y,d_Y)$, $f$ be the same as in 
 Theorem~\ref{thm:resolvent-minimax}, 
 $\{\lambda_n\}$ a sequence of positive real numbers 
 such that $\sum_{n=1}^{\infty}\lambda_n=\infty$, 
 and $\{(x_n,y_n)\}$ the sequence in $X\times Y$ 
 defined by $(x_1,y_1)\in X\times Y$ and 
 $(x_{n+1},y_{n+1})$ is the unique saddle point of 
 the function 
 \begin{align}
  (z,w) \mapsto 
  f(z,w) -\frac{1}{2\lambda _n}d_X(z,x_n)^2
 + \frac{1}{2 \lambda _n}d_Y(w,y_n)^2
 \end{align} 
 on $X\times Y$ for all $n\in \N$. 
 Then the following hold: 
 \begin{enumerate}
  \item[(i)] $\Saddle (f)$ is nonempty if and only if 
 $\{(x_n,y_n)\}$ is bounded; 
  \item[(ii)] if $\Saddle (f)$ is nonempty 
 and $\sum_{n=1}^{\infty}\lambda_n^2=\infty$, 
 $\{(x_n,y_n)\}$ is $\Delta$-convergent to 
 an element of $\Saddle (f)$.  
 \end{enumerate}
\end{theorem}

\begin{proof}
 Let $R_{\lambda_n}$ be the resolvent of $\lambda_n f$ 
 for all $n\in \N$ and $d$ the $\ell^2$-metric on $X\times Y$. 
 Then we have 
 \begin{align}
  (x_{n+1},y_{n+1})=R_{\lambda_n}(x_n,y_n)				  
 \end{align}
 for all $n\in \N$. 

 We first prove the part~(i). 
 Suppose that $\Saddle (f)$ is nonempty. 
 Let $(u,v)$ be an element of $\Saddle (f)$. 
 It follows from~(i) and~(iv) of Theorem~\ref{thm:resolvent-saddle} that 
 \begin{align}
  d\bigl((u,v), (x_{n+1},y_{n+1})\bigr) 
  =d\bigl((u,v), R_{\lambda_n}(x_{n},y_{n})\bigr) 
  \leq d\bigl((u,v), (x_{n},y_{n})\bigr) 
 \end{align}
 for all $n\in \N$ and hence $\{(x_n,y_n)\}$ is bounded. 
 Conversely, suppose that $\{(x_n,y_n)\}$ is bounded 
 and let $\mu>0$ be given. 
 Let $g\colon X\times Y\to \R$ 
 be a function defined by 
 \begin{align}
  g(x,y) = \limsup_{n\to \infty} 
  \frac{1}{\sum_{k=1}^n \lambda_k}
  \sum_{l=1}^n \lambda_l d\bigl((x_{l+1},y_{l+1}), (x,y)\bigr)^2
 \end{align}
 for all $(x,y)\in X\times Y$. 
 Theorem~\ref{thm:Kimura-K-bdd} ensures that $g$ has a unique 
 minimizer $(p,q)$. 
 Using~(iii) of Theorem~\ref{thm:resolvent-saddle}, we have 
 \begin{align}
 \begin{split}
  (\lambda_n+\mu)
  &d\bigl(
  (x_{n+1},y_{n+1}),R_{\mu}(p,q)
  \bigr)^2 \\
  &\quad \leq 
 \lambda_n
  d\bigl(
  (x_{n+1},y_{n+1}),(p,q)
  \bigr)^2
+ \mu 
  d\bigl(
  (x_{n},y_{n}),R_{\mu}(p,q)
  \bigr)^2  
 \end{split}
 \end{align}
 and hence 
 \begin{align}
 \begin{split}
  &\frac{1}{\sum_{k=1}^n\lambda_k}
  \sum_{l=1}^n\lambda_l
  d\bigl(
  (x_{l+1},y_{l+1}),R_{\mu}(p,q)
  \bigr)^2 \\
  &\leq 
  \frac{1}{\sum_{k=1}^n\lambda_k}
  \sum_{l=1}^n\lambda_l
  d\bigl(
  (x_{l+1},y_{l+1}),(p,q)
  \bigr)^2
+ \frac{\mu}{\sum_{k=1}^n\lambda_k}
d\bigl(
  (x_{1},y_{1}),R_{1}(p,q)
  \bigr)^2  
 \end{split}
 \end{align}
 for all $n\in \N$. Since $\sum_{k=1}^{\infty}\lambda_k=\infty$, 
 we have 
 \begin{align}
  g\bigl(R_{\mu}(p,q)\bigr)
  \leq g(p,q).
 \end{align}
 Consequently, we obtain $R_{\mu}(p,q)=(p,q)$. 
 Thus $(p,q)$ is an element of $\Saddle (f)$. 

 We finally prove the part~(ii). 
 Suppose that 
 $\Saddle (f)$ is nonempty 
 and $\sum_{n=1}^{\infty}\beta_n^2=\infty$. 
 Let $(u,v)$ 
 be an element of $\Saddle (f)$. 
 It follows from~(ii) of Theorem~\ref{thm:resolvent-saddle} that 
 \begin{align}
  \begin{split}
  &d\bigl((u,v),(x_{n},y_{n})\bigr)^2 \\
  &\geq d\bigl((u,v),(x_{n+1},y_{n+1})\bigr)^2 
  + d\bigl((x_{n+1},y_{n+1}),(x_{n},y_{n})\bigr)^2 \\
  &\quad +2\lambda_n 
  \bigl(
  f(u,y_{n+1}) -f(x_{n+1},v)
  \bigr) \\
  &\geq d\bigl((u,v),(x_{n+1},y_{n+1})\bigr)^2 
  + d\bigl((x_{n+1},y_{n+1}),(x_{n},y_{n})\bigr)^2 
  \end{split}
 \end{align}
 and hence $\{d\bigl((u,v),(x_{n},y_{n})\bigr)\}$ 
 is convergent and 
 \begin{align}
  \begin{split}
  &\sum_{n=1}^{\infty} d\bigl((x_{n+1},y_{n+1}),(x_{n},y_{n})\bigr)^2 \\
  &\quad \leq d\bigl((u,v),(x_{1},y_{1})\bigr)^2 
 -\lim_{n\to \infty} d\bigl((u,v),(x_{n},y_{n})\bigr)^2. 
  \end{split}
 \end{align}
 This implies that 
 \begin{align}\label{eq:thm:PPA-Saddle-a0}
  \lim_{n\to \infty}d\bigl((x_{n+1},y_{n+1}),(x_{n},y_{n})\bigr)= 0 
 \end{align}
 and 
 \begin{align}\label{eq:thm:PPA-Saddle-a1}
 \sum_{n=1}^{\infty} 
  \lambda_n^2 \left\{
\frac{1}{\lambda_n} d\bigl((x_{n+1},y_{n+1}),(x_{n},y_{n})\bigr) 
  \right\}^2
  <\infty
 \end{align}
 Since $\sum_{n=1}^{\infty}\lambda_n^2=\infty$, we know that 
 \begin{align}
  \liminf_{n\to \infty} 
\frac{1}{\lambda_n} d\bigl((x_{n+1},y_{n+1}),(x_{n},y_{n})\bigr) 
 =0. 
 \end{align} 
 On the other hand, using~(v) of Theorem~\ref{thm:resolvent-saddle}, 
 we also have 
 \begin{align}
 \begin{split}
  \frac{1}{\lambda_{n+1}} 
  d\bigl((x_{n+2},y_{n+2}),(x_{n+1},y_{n+1})\bigr)
 &=\frac{1}{\lambda_{n+1}} 
  d\bigl(R_{\lambda_{n+1}}R_{\lambda_n}(x_{n},y_{n}),
  R_{\lambda_n}(x_{n},y_{n})\bigr) \\
 &\leq 
    \frac{1}{\lambda_{n}} 
  d\bigl(R_{\lambda_n}(x_{n},y_{n}),(x_{n},y_{n})\bigr) \\
 &= \frac{1}{\lambda_{n}} 
  d\bigl((x_{n+1},y_{n+1}),(x_{n},y_{n})\bigr)  
 \end{split}
 \end{align}
 for all $n\in \N$ and hence the sequence 
 \begin{align}
  \left\{\frac{1}{\lambda_{n}}d\bigl((x_{n+1},y_{n+1}),(x_{n},y_{n})\bigr)\right\}
 \end{align}
 is convergent. Thus we have from~\eqref{eq:thm:PPA-Saddle-a1} that 
 \begin{align}\label{eq:thm:PPA-Saddle-a}
\lim_{n\to \infty} 
\frac{1}{\lambda_n} d\bigl((x_{n+1},y_{n+1}),(x_{n},y_{n})\bigr) 
 =0. 
 \end{align} 
 Let $(p,q)$ be an element of $\omega_{\Delta}\bigl(\{(x_n,y_n)\}\bigr)$. 
 Then there exists a subsequence $\{(x_{n_i},y_{n_i})\}$ of $\{(x_n,y_n)\}$ 
 which is $\Delta$-convergent to $(p,q)$. 
 It then follows from~\eqref{eq:thm:PPA-Saddle-a0} that 
 $\{(x_{n_i+1},y_{n_i+1})\}$ is also $\Delta$-convergent to $(p,q)$. 
 Let $(x,y)\in X\times Y$ be given. 
 Using~(ii) of Theorem~\ref{thm:resolvent-saddle}, we have 
 \begin{align}
 \begin{split}
  0
  &\leq f(x_{n+1},y)-f(x,y_{n+1}) +\frac{1}{2\lambda_n}\bigl\{d\bigl((x_n,y_n),(u,v)\bigr)^2 \\
  &\quad -d\bigl((x_{n+1},y_{n+1}),(x,y)\bigr)^2 
 -d\bigl((x_{n+1},y_{n+1}),(x_n,y_n)\bigr)^2\bigr\} \\
  &\leq f(x_{n+1},x)-f(y,y_{n+1}) 
 +\frac{1}{\lambda_n}d\bigl((x_{n+1},y_{n+1}),(x_n,y_n)\bigr)
d\bigl((x,y),(x_{n+1},y_{n+1})\bigr)  
 \end{split}
 \end{align}
 for all $n\in \N$. 
 Since $f(\,\cdot,v)-f(u,\cdot\,)$ is $\Delta$-upper semicontinuous on $X\times Y$, 
 we have from~\eqref{eq:thm:PPA-Saddle-a} that 
 \begin{align}
 \begin{split}
  0
  &\leq \limsup_{i\to \infty} 
  \Bigl[
 f(x_{n_i+1},y)-f(x,y_{n_i+1})  \\
  &\quad +\frac{1}{\lambda_{n_i}}d\bigl((x_{n_i+1},y_{n_i+1}),(x_{n_i},y_{n_i})\bigr)
d\bigl((u,v),(x_{n_i+1},y_{n_i+1})\bigr)
  \Bigr] \\
  &\leq f(p,y)-f(x,q).   
 \end{split}
 \end{align}
 This means that $(p,q)$ is an element of $\Saddle(f)$ 
 and hence $\{d\bigl((p,q),(x_{n},y_{n})\bigr)\}$ 
 is convergent. 
 Lemma~\ref{lem:Kimura-K-omega} implies that 
 $\{(x_n,y_n)\}$ is $\Delta$-convergent 
 a point $(x_{\infty},y_{\infty})$ 
 in $X\times Y$. 
 Consequently, we have 
 \begin{align}
  \{(x_{\infty},y_{\infty})\} =\omega_{\Delta}\bigl(\{(x_n,y_n)\}\bigr) 
  \subset \Saddle (f). 
 \end{align}
 Therefore the proof is completed. 
\end{proof}

\section*{Acknowledgements}
The author wishes to thank 
Ms.\ Kazuyo Hashimoto and Sakura Lemon McCartney in
Cafe Summer City, 
Hadano, 
Kanagawa, Japan 
for their unwavering encouragement and support.

\end{document}